\documentclass[12pt,leqno]{amsart}
\usepackage{amsmath}
\usepackage{amssymb} 
\usepackage{mathrsfs}
\usepackage[all]{xy}
\usepackage{color}
\usepackage{soul}
\usepackage[hidelinks]{hyperref}
\usepackage{enumitem}
\usepackage{graphicx}

\newcommand\R{{\mathbb R}}
\newcommand\T{{\mathbb T}}

\newcommand\Sp{{\mathbb S}}

\newcommand\Comp{{\mathbb C}}

\def\AA{{\mathcal A}}
\def\BB{{\mathcal B}}
\def\CC{{\mathcal C}}

\def\EE{{\mathcal E}}

\def\LL{{\mathcal L}}

\def\NN{{\mathcal N}}
\def\OO{{\mathcal O}}
\def\PP{{\mathcal P}}

\def\RR{{\mathcal R}}
\def\SS{{\mathcal S}}

\def\UU{{\mathcal U}}
\def\VV{{\mathcal V}}

\def\VV{{\mathcal V}}

\def\BBB{{\mathscr B}}
\def\CCC{{\mathscr C}}

\def\Rho{\textnormal{P}}

\def\eps{{\varepsilon}}
\def\vph{{\varphi}}

\def\spdis{\Sigma_{\textnormal{d}}}
\newcommand{\ah}{\mathbf{a}}
\newcommand{\beh}{\mathbf{b}}

\newtheorem{theo}{Theorem}
\newtheorem{prop}{Proposition}
\newtheorem{lem}[prop]{Lemma}

\newtheorem{rem}[prop]{Remark}
\newtheorem{rems}[prop]{Remarks}

\newtheorem{notation}[prop]{Notations}

\newcommand{\step}[2]{\medskip\noindent\textit{Step #1: #2.}}

\newcommand{\nul}{\textnormal{N}}
\newcommand{\dom}{\textnormal{D}}
\newcommand{\ran}{\textnormal{R}}
\newcommand{\Span}{\textnormal{Span}}

\newcommand{\real}{\Re \textnormal{e} \,}
\newcommand{\ima}{\Im \textnormal{m} \,}
\newcommand{\Id}{\textnormal{Id} \,}

\newcommand{\lin}{\LL}
\newcommand{\spc}{\EE}
\newcommand{\reg}{\AA}
\newcommand{\dis}{\BB}

\newcommand{\ling}{L}
\newcommand{\spcg}{E}

\newcommand{\linp}{\mathbf{L}}
\newcommand{\spcp}{\mathbf{E}(k)}

\title[Spectral study of the $L^2$ linearized Botlzmann operator]{Spectral study of the linearized Boltzmann operator in $L^2$ spaces with polynomial and Gaussian weights.}

\begin{document}
	
	\author{{\sc Pierre Gervais}}
	
	\date\today
	
	\maketitle
	
	\begin{abstract}
		The aim of this paper is to extend to the spaces $L^2\left(\R^d, (1+|v|)^{2k} dv\right)$ the spectral study \cite{EP} led in $L^2(\R^d, \exp(|v|^2/2) dv)$ on the space inhomogeneous linearized Boltzmann operator for hard spheres.
		More precisely, we look at the Fourier transform in the space variable of the inhomogeneous operator and consider the dual Fourier variable as a fixed parameter. We then perform a precise study of this operator for small frequencies (by seeing it as a perturbation of the homogeneous one) and also for large frequencies from spectral and semigroup point of views.
		Our approach is based on perturbation theory for linear operators as well as enlargement arguments from \cite{GMM}.
		%More precisely, (i) we prove the existence of a spectral gap for the inhomogeneous operator in the Fourier space, where the dual Fourier variable is considered to be a fixed parameter, (ii) we study the asymptotic behaviour of the isolated eigenvalues on the right of this gap and corresponding eigenspaces, when the frequency is close to zero, (iii) we prove, for low frequencies, an exponential decay for the part of the semigroup corresponding to the rest of the spectrum, and, for high frequencies, an exponential decay for the whole semigroup.
	\end{abstract}
	
	\vspace{1cm}
	
	\tableofcontents
	
	%%%%%%%%%%%%%%%%%%%%%%%%%%%%%%%%%%%%%%%%%%%%%%%%%%%%%%%%%%%%%%%%
	
	\section{Introduction} 
	\setcounter{equation}{0}
	\setcounter{prop}{0}
	
	%%%%%%%%%%%%%%%%%%%%%%%%%%%%%%%%%%%%%%%%%%%%%%%%%%%%%%%%%%%%%%%%
	
	\subsection{The model} \label{subsec:model}
	
	Consider a rarefied gas whose average number of particles located at position $x \in \Omega$, traveling at velocity $v \in \R^d$ at time $t \geq 0$ is given by~$F(t, x, v)$, where $\Omega=\T^d$ or $\R^d$, and $d \geq 2$. Assume furthermore that the particles are uncorrelated, and that they undergo hard sphere collisions where the energy and momentum are conserved. Finally, we assume these binary collisions are the only interactions between particles. Under these conditions, this density satisfies the \textit{Boltzmann equation}
	\begin{equation}
	\label{eqn:boltz}
	\tag{B}
	\partial_t F + v \cdot \nabla_x F = Q(F,F),
	\end{equation}
	which is a transport equation whose source term takes into account the binary collisions between the particles. The operator $Q$ is called the \textit{Boltzmann operator} or \textit{collision operator} and is an integral bilinear operator defined as
	\begin{equation*}
	Q(F,G)(v) := \int_{\R^d_{v_*}} \int_{\Sp^{d-1}_\sigma} |v-v_*|(F' G'_* - 
	F G_*) dv_* d\sigma
	\end{equation*}
	where we used the standard notations
	\begin{itemize}
		\item $v$ and $v_*$ for the velocities of two particles after the collision,
		\item $v'$ and $v_*'$ for their velocities before the collision, given by
		$$v' = \frac{v + v_*}{2} + \frac{|v-v_*|}{2} \sigma, ~ v_*' = \frac{v + v_*}{2} - \frac{|v-v_*|}{2} \sigma,$$
		\item $F':=F(v')$, $G_*':=G(v_*')$ and $G_*:=G(v_*)$.
	\end{itemize}
	
	\subsubsection{Equilibria}
	
	The (global) Maxwellian distributions, which write
	\begin{equation*}
	M_{\rho, \theta, u}(v) := \frac{\rho}{(2 \pi \theta)^{d/2}} \exp \left(-\frac{|v - u|^2}{2 \theta}\right)
	\end{equation*}
	for some $\rho, \theta > 0$ and $u \in \R^d$ can be shown to be equilibria of \eqref{eqn:boltz}. We will denote in this paper the normal centered distribution ($\rho=1, u=0, \theta=1$) by $M$.
	
	\subsubsection{Hydrodynamic limits}
	
	By choosing a system of reference values for length, time and velocity (see for example \cite{G}), we get a dimensionless version of the equation:
	\begin{equation*}
	\eps \partial_t F^\eps+v \cdot \nabla_x F^\eps = \frac{1}{\eps} Q(F^\eps, F^\eps),
	\end{equation*}
	where $\eps$ is the Knudsen number and corresponds to the mean free path, that is to say the average distance traveled by a particle between two collisions. Performing the linearization $F^\eps =: M + \eps f^\eps$, the equation rewrites in terms of $f^\eps$ as
	\begin{equation}
	\label{eqn:boltzmann_scaled_lin}
	\eps \partial_t f^\eps + v \cdot \nabla_x f^\eps = \lin f^\eps + \frac{1}{\eps}Q(f^\eps,f^\eps),
	\end{equation}
	where $\lin h := Q(M, h) + Q(h, M)$. Letting $\eps$ go to zero, we expect to get the dynamics of a fluid as the amount of collisions will then go to infinity. This issue of unifying the mesoscopic and macroscopic points of view goes back to Hilbert, and several formal methods have been suggested by Hilbert \cite{Hilbert}, Chapman, Enskog and Grad \cite{Grad}. These were made rigorous by C. Bardos, F. Golse and D. Levermore in \cite{BGL} by proving that if~$f^\eps$ and its moments converge in some weak sense as $\eps$ goes to zero, then the limiting moments satisfy the Navier-Stokes-Fourier system.
	In their paper \cite{BU}, C. Bardos and S. Ukai showed that these convergence assumptions hold for any initial datum $f_{\text{in}}$ small enough in some norm by rewriting the previous equation as an integral one:
	\begin{equation*}
	f^\eps(t) = S^\eps\left(\eps^{-2} t\right) f_{\text{in}} + \frac{1}{\eps} \int_0^t S^\eps\left(\eps^{-2} s\right) Q(f^\eps(t-s), f^\eps(t-s)) ds,
	\end{equation*}
	where $S^\eps$ is the semigroup generated by $\lin-\eps v \cdot \nabla_x$. Its dynamics as $\eps$ goes to zero has been described by the spectral study \cite{EP} of the inhomogeneous linearized Boltzmann operator in Fourier space, i.e. $\lin-i \eps v \cdot \xi$, where $\xi$ is the dual variable of~$x$. These results were used by I. Gallagher and I. Tristani in \cite{IGT} to prove a partial converse result: for any solution to the Navier-Stokes-Fourier system defined on a time interval~$[0, T]$, the solution $f^\eps$ to \eqref{eqn:boltzmann_scaled_lin} exists for $\eps$ small enough (depending on the Navier-Stokes-Fourier solution), is defined on $[0, T]$, and converges to some limit~$f^0$, whose moments are the aforementioned solution of the Navier-Stokes-Fourier system.
	
	\subsection{Statement of the main results}
	
	\label{scn:goals}
	
	Let us define some notations used in the statement of Theorem \ref{thm:intro_spectral} and \ref{thm:intro_splitting_semigroup}. We denote~$L^2(m)$ the $L^2$ Hilbert space associated with the measure $m^2(v)dv$. $\BBB(X, Y)$ is the space of bounded linear operators from a Banach space $X$ to another one $Y$. For a linear operator $\Lambda$, we denote by~$\Sigma(\Lambda)$ its spectrum. If it generates a strongly continuous semigroup, we denote it by~$S_\Lambda$, and~$\Pi_{\Lambda, \lambda}$ is the spectral projector associated with an eigenvalue $\lambda \in \spdis(\Lambda)$, where~$\spdis(\Lambda)$ is the discrete spectrum of~$\Lambda$. Finally, we write $\Delta_a := \{\real z > a\}$.
	
	We will show that, similarly to the results in \cite{EP}, when considered as a closed operator in $L^2\left(M^{-1/2}\right)$ or $L^2\left( \langle v \rangle^k \right)$, the semigroup generated by $\lin - i v \cdot \xi$ has exponential decay in time (for large frequencies~$\xi$), and splits (for small frequencies~$\xi$) into a first part corresponding to its rightmost eigenvalues, and a remainder that decays exponentially in time. We also gather enough information on the eigenvalues, spectral projectors and remainder so that we expect that \cite{BU} or \cite{IGT} may be adapted to~$L^2\left(\langle v \rangle^k\right)$. This analysis is postponed to a future work. Let us present the main theorems of this paper.
	
	\newcommand{\thpart}[2]{
		\bigskip
		\noindent{\textnormal{\textbf{(#1) - #2}}}
	}
	
	\begin{theo}
		\label{thm:intro_spectral}
		There exists $k_* > 2$ such that for any fixed $k > k_*$, denoting the spaces~$\spcg := L^2\left(M^{-1/2}\right)$, $\spcp := L^2 \left(\langle v \rangle^k\right)$, the operator $\lin_\xi := \lin - iv \cdot \xi$ is closed in both spaces $\spcp$ and $\spcg$ for any $\xi \in \R^d$. Furthermore, the following holds:
		
		\thpart{1}{Spectral gaps and expansion of the eigenvalues.}There exist $\ah,  \beh, r_0 > 0$ such that, in both spaces,
		\begin{gather}
		\label{eqn:spectral_gap_small_freq}
		\Sigma(\lin_\xi) \cap \Delta_{-\ah} = \{ \lambda_{-1}(|\xi|), \dots, \lambda_2(|\xi|) \} \subset \spdis(\lin_\xi), ~ |\xi| \leq r_0, \\
		\label{eqn:spectral_gap_large_freq}
		\Sigma(\lin_\xi) \cap \Delta_{-\beh} = \emptyset, ~ |\xi| \geq r_0,
		\end{gather}
		where the eigenvalues $\lambda_j$ have the expansion around $|\xi| = 0$ in $\Comp$
		\begin{equation}
		\label{eqn:eigenvalue_expansion}
		\lambda_j(|\xi|) = \lambda_j^{(1)} |\xi| + \lambda_j^{(2)} |\xi|^2 + o(|\xi|^2),
		\end{equation}
		with $\lambda_{\pm 1}^{(1)} = \pm i \sqrt{1+2/d}$, $\lambda_j^{(1)} =  0$ for $j=0, 2$, and $\lambda_j^{(2)} < 0$ for $j=-1,\dots,2$.
		
		\thpart{2}{Spectral decomposition and expansion of the projectors.} There exist projectors $\PP_j(\xi) \in \BBB(\spcp, \spcg)$ for any~$0 < |\xi| \leq r_0$ and $j=-1,\dots, 2$, that expand in~$\BBB(\spcp, \spcg)$ as
		\begin{equation}
		\label{eqn:projector_expansion}
		\PP_j(\xi) = \PP_j^{(0)}\left(\widetilde{\xi}\right) + |\xi| \PP_j^{(1)}\left(\widetilde{\xi}\right) + \PP_j^{(2)}(\xi), ~ \widetilde{\xi} := \frac{\xi}{|\xi|},
		\end{equation}
		where $\left\|\PP_j^{(\ell)}\right\|_{\BBB(\spcp, \spcg)}$ is uniformly bounded in $|\xi|$, and $\left\|\PP_j^{(2)}(\xi)\right\|_{\BBB(\spcp, \spcg)} =~o(|\xi|)$.
		
		\medskip
		\noindent
		For $j=0, \pm 1$, $\PP_j^{(0)}\left(\widetilde{\xi}\right)$ is a projection onto $\Comp e_j^{(0)} \left(\widetilde{\xi}\right)$, with
		\begin{gather}
		\label{eqn:e_0_0}
		e_0^{(0)} \left(\widetilde{\xi}\right) = \left(1 - \frac{1}{2} \left(|v|^2-d\right) \right)M,\\
		\label{eqn:e_pm1_0}
		e_{\pm 1}^{(0)} \left(\widetilde{\xi}\right) = \left(1 \pm \widetilde{\xi} \cdot v + \frac{1}{d} \left(|v|^2-d\right) \right) M,
		\end{gather}
		and $\PP_2^{(0)}\left(\widetilde{\xi}\right)$ is a projection on $\left\{ c \cdot v M ~  | ~ c \cdot \widetilde{\xi} = 0 \right\}$, which is spanned by
		\begin{gather}
		e_{2,\ell}^{(0)}\left(\widetilde{\xi}\right) = C_\ell\left(\widetilde{\xi}\right) \cdot vM, ~ \ell=1,\dots,d-1,
		\end{gather}
		where $\left(\widetilde{\xi}, C_1\left(\widetilde{\xi}\right), \dots, C_{d-1}\left(\widetilde{\xi}\right)\right)$ can be assumed to be any fixed orthonormal basis of~$\R^d$. Furthermore, they satisfy
		\begin{gather}
		\label{eqn:lin_xi_p_j_xi}
		\lin_\xi \PP_j(\xi) = \lambda_j(|\xi|) \PP_j(\xi), \\
		\label{eqn:proj_orthogonal}
		\PP_j(\xi) \PP_\ell(\xi) = 0, ~ j \neq \ell,\\
		\label{eqn:sum_proj_0_pi_0}
		\sum_{j = -1}^2 \PP_j^{(0)}\left(\widetilde{\xi}\right) = \Pi_{\lin, 0}, ~ \widetilde{\xi} \in \Sp^{d-1}.
		\end{gather}
		
		\thpart{3}{Expression of the projectors.}
		For any $0 < |\xi| \leq r_0$, $j=-1,\dots,1$ and~$\ell = 1, \dots, d-1$, there exist functions $e_j(\xi), e_{2, \ell}(\xi) \in \spcg$ and $f_j(\xi), f_{2, \ell}(\xi) \in \spcp$ such that the projectors write
		\begin{gather}
		\label{eqn:expression_p_j}
		\PP_j(\xi) g = \langle g, f_{j}(\xi) \rangle_{\spcp} \, e_j(\xi), ~ j=0, \pm 1, \\
		\label{eqn:expression_p_2}
		\PP_2(\xi) g = \sum_{\ell=1}^{d-1}\left\langle g, f_{2, \ell}(\xi) \right\rangle_{\spcp} \, e_{2, \ell}(\xi), \\
		\label{eqn:biorthogonal}
		\left\langle e_\alpha(\xi), f_\beta(\xi)\right\rangle_{\spcp} = \delta_{\alpha, \beta},
		\end{gather}
		where $\alpha$ and $\beta$ are any indices among $-1, 0, 1, (2, 1), \dots, (2, d-1)$, and they have the following expansions:
		\begin{gather}
		\label{eqn:expansion_e_j}
		e_\alpha(\xi) = e_\alpha^{(0)}\left(\widetilde{\xi}\right) + |\xi| e_\alpha^{(1)}\left(\widetilde{\xi}\right) + e^{(2)}_\alpha(\xi),\\
		\label{eqn:expansion_f_j}
		f_\alpha(\xi) = f_\alpha^{(0)}\left(\widetilde{\xi}\right) + |\xi| f_\alpha^{(1)}\left(\widetilde{\xi}\right) + f_\alpha^{(2)}(\xi),
		\end{gather}
		where $\alpha$ is any index among $-1, 0, 1, (2, 1), \dots, (2, d-1)$, $\left\|e_\alpha^{(\ell)}\right\|_{\spcg}$, $\left\|f_\alpha^{(\ell)}\right\|_{\spcp}$ are uniformly bounded in $|\xi| \in (0, r_0]$, $\left\|e_\alpha^{(2)}(\xi)\right\|_\spcg = o(|\xi|)$ and $\left\|f_\alpha^{(2)}(\xi)\right\|_{\spcp} =~o(|\xi|)$.
	\end{theo}
	
	\begin{rems}
		A few precisions are to be made on these results.
		\begin{itemize}
			\item In this theorem, $e_\alpha : \{0 < |\xi| \leq r_0\} \rightarrow \spcg$, $f_\alpha : \{0 < |\xi| \leq r_0\} \rightarrow \spcp$ and~$\PP_j : \{0 < |\xi| \leq r_0\} \rightarrow \BBB(\spcp, \spcg)$ are measurable.
			\item Using the fact that in the Hilbert space $\spcg$, $\left(\lin_\xi\right)^* = \lin_{-\xi}$, and the relation~$O \lin_\xi O^{-1} = \lin_{O^{-1} \xi}$, where $O$ is any real $d \times d$ orthogonal matrix, one can show that $\lambda_1$ and $\lambda_{- 1}$ are conjugate to one another, and $\lambda_j$ is real for~$j=0,2$.
			\item Furthermore, one can deduce an expression of $f_j(\xi)$ in terms of $e_j(\xi)$ using the fact that $\PP_2(\xi)^*=\PP_2(-\xi)$, $\PP_j(\xi)^*=\PP_{-j}(-\xi)$ for~$j=0, \pm 1$, and using the relation
			$$\langle f, g \rangle_\spcg = \langle f, g \langle v \rangle^{-2k}M^{-1} \rangle_{\spcp}.$$
		\end{itemize}
		
		%	One could even assume them to be smooth on $\{|\xi| \leq r_0\} - \{0\}^2 \times [0, r_0]$ by considering in the proof some fixed family of smooth real $3 \times 3$ orthogonal matrices $\left(O_\omega\right)_{\omega \in \Sp^{d-1} - \{(0, 0, 1)\}}$ satisfying $O_\omega \omega=(1,0,0)$.
	\end{rems}
	
	%
	%The closedness of $\lin_\xi$ is proved in Section \ref{scn:closedness_decomposition}. The existence of spectral gaps are proved in Section \ref{scn:spetral_gaps} and the existence of eigenvalues, together with the expansion of the projectors and their expression, in Section \ref{scn:eigen_problem}.
	
	\begin{theo}
		\label{thm:intro_splitting_semigroup}
		Under the same assumptions, denoting $\spc=\spcg$ or $\spcp$, there exists constants $C > 0$ and $\gamma \in (0, \ah)$ such that for any $\xi \in \R^d$, $\lin_\xi$ generates on $\spc$ a~$\CC^0$-semigroup that splits as
		\begin{align}
		\label{eqn:splitting_semigroup}
		S_{\lin_{\xi} }(t) = \chi(\xi) \sum_{j=-1}^2 e^{t \lambda_j(\xi) } \PP_j\left(\xi\right) + \VV(t, \xi), ~ \xi \neq 0, \\
		S_\lin(t) = \sum_{j=-1}^2 \PP_j^{(0)}\left(\widetilde{\xi}\right) + \VV(t, 0) = \Pi_{\lin, 0} + \VV(t, 0), ~ \widetilde{\xi} \in \Sp^{d-1},
		\end{align}
		where we have denoted $\chi$ the characteristic function of $\{|\xi| \leq r_0\}$, and the remainder~$\VV$ satisfies
		\begin{gather}
		\label{eqn:orthogonal_proj_remainder}
		\PP_j(\xi) \VV(t, \xi) = \VV(t, \xi) \PP_j(\xi) = 0,\\
		\label{eqn:exponential_decay}
		\|\VV(t, \xi)\|_{\BBB(\spc)} \leq C e^{-\gamma t}.
		\end{gather}
	\end{theo}
	
	\subsection{Method of proof and state of the art}
	
	Theorem \ref{thm:intro_spectral} was initially proved in \cite{EP} in the space $L^2\left(M^{-1/2}\right)$. The authors first proved that for some $\delta > 0$, the following equations are equivalent for $|\xi|, |\lambda| \leq \delta$:
	\begin{gather*}
	\left(\lin - i v \cdot \xi - \lambda\right)f = 0,\\
	F(\lambda, \xi) f_0 = 0,
	\end{gather*}
	where $f_0$ is the projection of $f$ on $\NN$, $f_1 := f- f_0$ is related to $f_0$ by $f_1 = G(\lambda, \xi)f_0$, and $F(\lambda, \xi) \in \BBB(\NN)$, $G(\lambda, \xi) \in \BBB(\NN, \NN^\bot)$ are smooth in $\xi$ and $\lambda$. They then proceed to solve $\det F(\lambda, \xi) = 0$ for $\lambda=\lambda(\xi)$ using the implicit function theorem, exhibit corresponding $f_0(\xi)$, construct the eigenfunctions and then the spectral projectors.
	
	In their proof, the threshold $\delta$ was not found using constructive estimates, nor do they prove the existence of a spectral gap for $|\xi|$ bounded away from zero. However, their results hold for a general class of potentials, including hard and Maxwellian potentials with cut-off.
	
	\smallskip
	
	T. Yang and H. Yu \cite{YY} have a similar approach and still prove their results in~$L^2\left(M^{-1/2}\right)$, but they cover a broader class of kinetic equations. Furthermore, they prove the existence of a spectral gap for large $\xi$ and encounter the same difficulties as in this paper: they are able to provide constructive estimates for small and large frequencies, but need a non-constructive argument to deal with intermediate ones.
	
	We also mention \cite{CIP} and \cite{UY} who prove most of these results with similar approaches.
	
	\smallskip
	
	In this paper, we generalize the results from \cite{EP} in spaces of the form $L^2\left(\langle v \rangle^k\right)$ using a new splitting of the homogeneous operator as well as an ``enlargement theorem'', both from \cite{GMM}. This splitting has the same properties in both Gaussian and polynomial spaces (dissipativity and relative boundedness, regularizing effect, see Lemma \ref{prop:def_lin_xi}) which allows to treat both cases in a unified framework, and the aforementioned ``enlargement theorem'' guaranties that the spectral properties (structure of the spectrum and eigenspaces) do not depend on the specific choice of space, be it Gaussian or polynomial. We can therefore rely on previous studies of the Gaussian case when convenient.
	
	As we deal with hard sphere case, the inhomogeneous operator in Fourier space can be seen as a relatively bounded perturbation of the homogeneous operator and thus be studied through classical (analytic) perturbation theory. In particular, all estimates are constructive, except for the exponential decay estimates for large frequencies.
	
	Unlike \cite{EP} and \cite{YY} who compute the roots of the dispersion relations associated with the linear inhomogeneous Boltzmann equation, we prove that for small $\xi$, the zero eigenvalue (resp. the null space $\NN$) of the homogeneous operator ``splits'' into several eigenvalues (resp. an invariant space $\NN(\xi)$ isomorphic to $\NN$). We then consider $\left(\lin_\xi\right)_{|\NN(\xi)} \in \BBB(\NN(\xi))$ and straighten $\NN(\xi)$ into $\NN$ to get a new operator~$\widetilde{\lin}(\xi) \in \BBB(\NN)$ conjugated to $\left(\lin_\xi\right)_{|\NN(\xi)}$ which we study using finite dimensional perturbation theory.

	\subsection{Outline of the paper}
	
	In Section 2, we show using results from \cite{GMM} that there exist some threshold~$k_* > 0$ such that in both spaces $L^2\left(M^{-1/2}\right)$ and $L^2\left( \langle v \rangle^k \right)$ with~$k > k_*$, $\lin_\xi$ generates a strongly continuous semigroup, satisfies some rotation invariance property and the multiplication operator by $v$ is $\lin_\xi$-bounded. Then, combining results from \cite{GMM} and \cite{UY}, we show the existence in both spaces of spectral gaps for small and large $\xi$: there exists~$\ah, \beh > 0$ such that for large $\xi$, the spectrum does not meet $\Delta_{-\beh}$, and for small~$\xi$, the part $\Sigma(\lin_\xi) \cap \Delta_{-\ah}$ contains a finite amount of discrete eigenvalues enclosed by some fixed path $\Gamma$.
	
	In Section 3, this path allows to transform the eigenvalue problem into an equivalent one on the finite dimensional null-space of $\lin$ and in turn derive expansions for the eigenvalues and associated spectral projectors, thus proving Theorem \ref{thm:intro_spectral}.
	
	In Section 4, we prove Theorem \ref{thm:intro_splitting_semigroup}. The splitting comes from Theorem \ref{thm:intro_spectral}, and the decay estimate from Theorem \ref{thm:gpg} whose assumptions are obtained using estimates from \cite{UY} combined with \cite{GMM}, and the continuity of the resolvent.
	
	We recall in the appendix some results from spectral theory and semigroup theory.
	
	\subsection{Notations and definitions}
	
	\subsubsection{Function spaces}
	For any Borel function $m > 0$ and $p \in [1, \infty]$, we define the space $L^p(m)$ as the set of measurable functions $f : \R^d \to \Comp$ such that
	$$\|f\|_{L^p(m)} := \|f m\|_{L^p} < \infty.$$
	
	\subsubsection{Operator theory}
	\label{scn:operator_theory}
	For some given Banach spaces $X$ and $Y$, we will denote the space of closed linear operators $\Lambda$ from their domain $\dom(\Lambda)$ to $Y$ by $\CCC(X,Y)$. The space of bounded linear operators will be denoted $\BBB(X,Y)$. For any linear operator~$\Lambda$, we denote its null space by $\nul (\Lambda)$ and its range by $\ran (\Lambda)$.
	
	\smallskip
	\noindent
	In particular, we write $\CCC(X) = \CCC(X , X)$ and $\BBB(X) = \BBB(X,X)$. We will also consider the \textit{resolvent set} $\Rho(\Lambda)$ of $\Lambda$ which is defined to be the open set of all $z \in \Comp$ such that $\Lambda-z$ is bijective from $\dom (\Lambda)$ onto $X$, and whose inverse is a bounded operator of $X$. The \textit{resolvent operator} is an analytic function defined by
	\begin{equation*}
	\begin{array}{lccl}
	\RR_\Lambda : &\Rho(\Lambda)& \rightarrow &\BBB(X) \\
	&z &\mapsto &(z - \Lambda)^{-1},
	\end{array}
	\end{equation*}
	and cannot be continued analytically beyond this set.
	The complement of $\Rho(\Lambda)$ is called the \textit{spectrum} of $\Lambda$ and is denoted $\Sigma(\Lambda) = \Comp - \Rho(\Lambda)$, which is therefore the set of all values $\lambda$ such that $\Lambda-\lambda$ is not boundedly invertible.
	\medskip\\
	When a spectral value $\lambda$ is isolated in the spectrum, or in other words when for some~$\eps > 0$ small enough
	$$\Sigma(\Lambda) \cap \{z \in \Comp ~| ~ |z-\lambda| < \eps\} = \{\lambda\},$$
	we may define the associated \textit{spectral projector}
	\begin{equation*}
	\Pi_{\Lambda, \lambda} := \frac{1}{2 i \pi} \int_\Gamma \RR_\Lambda(z) dz = \text{Res}\left(\RR_\Lambda ; \lambda\right),
	\end{equation*}
	where $\Gamma$ is some closed path, encircling $\lambda$ and only $\lambda$ exactly once, and that does not meet the spectrum (a circle or any closed loop that can be continuously stretched within $\Rho(\Lambda)$ into a circle). It is well known that this operator is well defined and is a projector whose range satisfies the following inclusion
	\begin{equation*}
	\nul (\Lambda - \lambda) \subset \ran\left(\Pi_{\Lambda, \lambda} \right).
	\end{equation*}
	We call the left-hand side the \textit{geometric eigenspace} and the right-hand side the \textit{algebraic eigenspace}, and their dimensions are called respectively the \textit{geometric} and \textit{algebraic multiplicities}.
	\medskip\\
	When the algebraic multiplicity is finite, i.e. $\dim \ran\left(\Pi_{\Lambda, \lambda}\right) < \infty$, $\nul(\Lambda - \lambda) \neq~\{0\}$ and the spectral value $\lambda$ is called a \textit{discrete eigenvalue}, which we write $\lambda \in \spdis(\Lambda)$. 
	
	\medskip
	\noindent
	We will also denote by $\OO\left(\R^d\right)$ the set of $d \times d$ real orthogonal matrices, and denote the action of $O \in \OO(\R^d)$ on any function $f$ defined on~$\R^d_v$ by
	$$(Of)(v) := f(Ov).$$
	In particular, if $\Phi$ is the multiplication operator by a function $\phi=\phi(v)$, then~$O \Phi O^{-1}$ is the multiplication operator by $O\phi$.
	
	\medskip
	
	\subsubsection{Semigroup theory}
	\label{scn:sg_theory}
	For any $a \in \R$, we write $\Delta_a := \{\real z > a\}$, and for any~$\CC^0$-semigroup generator $\Lambda$, we write its semigroup $S_\Lambda(t)$.
	
	\medskip
	
	\section{General properties of the linearized operator}
	\setcounter{equation}{0}
	\setcounter{prop}{0}
	
	The linearized operator $\lin$ has been extensively studied in the space $L^2(M^{-1/2})$ by Hilbert \cite{Hilbert} and Grad \cite{Grad}, let us recall its main properties.
	
	\begin{theo}
		\label{thm:hilbert_decomposition}
		Denote $E = L^2(M^{-1/2})$ and $L = \lin_{|E}$. The operator $L$ is closed in~$E$, self-adjoint, dissipative and densely defined. It splits as
		\begin{equation}
		\label{eqn:hilbert_decomposition}
		L = -\nu + K,
		\end{equation}
		where $K$ is compact on $E$ and $\nu$ is a continuous function of $v \in \R^d$ defined by
		$$\nu(v) := \int_{\R^d_{v_*} \times \Sp^{d-1}_\sigma} M_* |v-v_*| dv_* d\sigma$$
		and satisfying for some $\nu_0, \nu_1 > 0$
		\begin{equation}
		\nu_0 \langle v \rangle \leq \nu(v) \leq \nu_1 \langle v \rangle.
		\end{equation}
		There exists a spectral gap for some $\ah_0 \in (0, \nu_0)$,
		$$\Sigma(L) \cap \Delta_{-\ah_0} = \{0\},$$
		where $\Delta_{-\ah_0} := \{\real z > -\ah_0\}$. The eigenvalue 0 is semi-simple and the null space of $L$, denoted $\NN$, is spanned by the following basis, orthogonal in $E$:
		\begin{align*}
		\begin{cases}
		\varphi_0(v) = M(v), \\
		\varphi_j(v) = v_j M(v),~ j=1,\dots,d, \\
		\displaystyle \varphi_{d+1}(v) = \left(|v|^2-d\right) M(v).
		\end{cases}
		\end{align*}
		Finally, for any $O \in \OO\left(\R^d\right)$, $O L = L O$.
	\end{theo}
	
	\begin{rem}
		The existence of this spectral gap has originally been proved using Weyl's theorem. However, C. Baranger and C. Mouhot provided in \cite[Theorem 1.1]{BM} an explicit estimate for $\ah_0$:
		\begin{equation*}
		\ah_0 \geq \frac{\pi}{48 \sqrt{2 e}} \cdot
		\end{equation*}
	\end{rem}
	
	Using this decomposition, R. Ellis and M. Pinsky \cite{EP} proved in the space $L^2\left(M^{-1/2}\right)$ the theorems stated in Section \ref{scn:goals}. However, this decomposition does not have the same nice properties in the larger spaces of the form $L^2\left(\langle v \rangle^k\right)$. M. Gualdani, S. Mischler and C. Mouhot \cite{GMM} introduced a new decomposition with similar properties which hold in both spaces and is presented in Lemma \ref{prop:def_lin_xi}.
	
	\subsection{Closedness and decomposition of \texorpdfstring{$\lin_{\xi}$}{Lξ}}
	\label{scn:closedness_decomposition}
	
	In this section, we present a decomposition of the linearized operator $\lin_\xi = \reg + \dis_\xi$,
	where in both spaces $L^2\left(M^{-1/2}\right)$ and $L^2\left( \langle v \rangle^k \right)$, $\reg$ boundedly maps its domain to
	$L^2\left(M^{-1/2}\right)$, $\dis_\xi + \ah$ is m-dissipative for some $\ah > 0$, and the multiplication operator by $v$ is
	$\lin_\xi$-bounded.
	
	The following lemma combines several results from \cite{GMM} that were used to prove the existence of a spectral gap for $\lin - v \cdot \nabla_x$ in a large class of Sobolev spaces $W^{s, p}_x W^{\sigma, q}_v$. We focus instead on $\lin_\xi$ in $L^2_v$ spaces, and also show the relative boundedness of the multiplication operator by $v$.
	
	\begin{lem}
		\label{prop:def_lin_xi}
		There exists some $k_* \geq 5/2$ such that for any $k > k_*$, the perturbed linearized Boltzmann operator splits as
		\begin{equation*}
		\lin_\xi = \dis_\xi + \reg = \dis - iv\cdot \xi + \AA,
		\end{equation*}
		where, denoting $\spc=L^2\left(M^{-1/2}\right)$ or $L^2\left( \langle v \rangle^k \right)$, the operator $\reg_{|\spc}$ is bounded from~$\spc$ to~$L^2\left(M^{-1/2}\right)$, $\left(\dis_\xi\right)_{|\spc}$ and thus $\left(\lin_\xi\right)_{|\spc}$ are closed in~$\spc$ with the common dense domain~$\{f \in \spc ~| ~ \nu f \in \spc \}$.
		
		\medskip
		\noindent
		Furthermore, there exist $C > 0$ and $\ah_1 \in (0, \nu_0)$ such that for any $\xi \in \R^d$
		\begin{gather*}
		\dis_\xi + \ah_1 \text{ is m-dissipative}, \\
		\|v f\|_\spc \leq C \left(\|\lin_{\xi} f\|_\spc + \|f\|_\spc\right), ~ f \in \dom(\lin_\xi).
		\end{gather*}
		Finally, for any $O \in \OO\left(\R^d\right)$,
		\begin{equation}
		\label{eqn:isometry_inv}
		O \lin_\xi = \lin_{O \xi}O.
		\end{equation}
		
	\end{lem}
	
	\begin{proof}
		In \cite[section 4.3.3]{GMM}, the authors introduce a new splitting of the linearized operator~$\lin$, which allows to deal with polynomial weights:
		\begin{equation*}
		\lin = \AA_\delta + \BB_\delta = \AA_\delta + \left(-\nu + \overline{\BB}_\delta\right), ~ \delta \in (0, 1).
		\end{equation*}
		Here, $\AA_\delta$ is an integral operator with smooth compactly supported kernel, and $\overline{\BB}_\delta$ satisfies by \cite[Lemma 4.12, (4.40)]{GMM} the estimate
		\begin{gather}
		\label{eqn:rel_bound_dis}
		\left\|\overline{\BB}_\delta f\right\|_{L^2\left( \langle v \rangle^{q} \nu^{-1/2}\right)} \leq b_\delta(q-1/2) \left\|f\right\|_{L^2\left( \langle v \rangle^{q} \nu^{1/2} \right)}, ~ q > 5/2, \\
		\notag
		b_\delta(q) \underset{\delta \rightarrow 0}{\longrightarrow} b(q) := \frac{4}{\sqrt{(q+1)(q-2)}}, \\
		\notag
		\left\|\overline{\BB}_\delta\right\|_{\BBB\left(L^2\left(M^{-1/2}\right)\right)} \underset{\delta \rightarrow 0}{\longrightarrow} 0.
		\end{gather}
		We consider in this proof $k_*$ such that $b(k_*-1/2) = 1$ and fix some~$k > k_*$ so that~$b(k-1/2) < 1$. We also consider $\delta > 0$ to be small enough so that 
		\begin{align*}
		\ah_1 &:= \nu_0 - \left\|\overline{\BB}_\delta\right\|_{\BBB(L^2\left(M^{-1/2}\right))} > 0,&& \text{if}~ \spc = L^2\left(M^{-1/2}\right),\\
		\ah_1 &:= \nu_0 \left(1 - b_\delta(k-1/2)\right) > 0,&& \text{if}~\spc = L^2\left( \langle v \rangle^k \right).
		\end{align*}
		
		\step{1}{Boundedness and closedness at $\xi=0$}	
		As $\AA_\delta$ is an integral operator with a bounded and compactly supported kernel, it is clear that for any of the two spaces~$\spc = L^2\left(M^{-1/2}\right), L^2\left( \langle v \rangle^{k} \right)$, this operator is bounded from $\spc$ to $L^2\left(M^{-1/2}\right)$.
		
		When $\spc=L^2\left(M^{-1/2}\right)$, $\BB_\delta$ is the sum of a closed and a bounded operator, so it is closed and densely defined.
		
		When $\spc=L^2\left( \langle v \rangle^k \right)$, note \footnote{see for instance Remark 4.1 of \cite{GMM}} that $\nu_0 > 1$, which combined with \eqref{eqn:rel_bound_dis} implies that~$\overline{\BB}_\delta$ is $\nu$-bounded, with $\nu$-bound equal to $b_\delta(k) < 1$. Hence $\BB_\delta$ is closed on~$L^2\left( \langle v \rangle^{k} \right)$ by \cite[Theorem IV-1.1]{Kato}.
		
		In both cases, $\BB_\delta$ and thus $\lin$ are closed and defined on the dense domain
		\begin{equation*}
		\dom(\BB_\delta)=\dom(\lin)=\{f \in \spc ~| ~ \nu f \in \spc\}.
		\end{equation*}
		
		\step{2}{Dissipativity estimates}
		By the definition of $\ah_1$, $\BB_\delta + \ah_1$ is dissipative on~$L^2\left(M^{-1/2}\right)$.
		
		\medskip
		\noindent
		In the polynomial space, we have from \eqref{eqn:rel_bound_dis} that
		\begin{equation*}
		\left| \left\langle \overline{\BB}_\delta f, f \right\rangle_{L^2\left( \langle v \rangle^k \right)} \right| \leq b_\delta(k-1/2) \left\|\nu^{1/2} f\right\|_{L^2\left(\langle v \rangle^k\right)}^2.
		\end{equation*}
		Thus, by the definition of $\ah_1$, we have
		\begin{align*}
		\langle \BB_\delta f, f \rangle_{L^2\left(\langle v \rangle^k\right)} \leq \left(1 - b_\delta(k)\right)\|\nu^{1/2} f\|_{L^2\left(\langle v \rangle^k\right)}^2 \leq -\ah_1 \|f\|_{L^2\left(\langle v \rangle^k\right)},
		\end{align*}
		which yields the dissipativity of $\BB_\delta+\ah_1$ on $L^2\left( \langle v \rangle^k \right)$.
		
		\medskip
		\noindent
		
		\step{3}{Relative bound and closedness of $\lin_\xi$ and $\BB_\xi$}
		First, let us show that $\nu+i v \cdot \xi$ is~$\lin_\xi$-bounded uniformly in~$\xi \in \R^d$:
		\begin{align*}
		\|(\nu+ i v \cdot \xi)f\|_\spc & \leq \|\BB_\xi f\|_\spc + \left\|\overline{\BB}_\delta f\right\|_\spc \\
		& \leq \|\BB_\xi f\|_\spc + \beta \|\nu f\|_\spc \\
		& \leq \|\BB_\xi f\|_\spc + \beta \|(\nu + iv \cdot \xi) f\|_\spc
		\end{align*}
		where $\beta = b_\delta(k)$ for $\spc=L^2 \left( \langle v \rangle^k \right)$, and $\beta=\left\|\overline{\BB}_\delta\right\|_{\BBB\left(L^2\left(M^{-1/2}\right)\right)}$ for $\spc=L^2\left(M^{-1/2}\right)$. In both cases, we assume $\delta$ to be small enough so that $\beta \in (0, 1)$, which allows to write
		\begin{align*}
		\|(\nu+i v \cdot \xi)f\|_\spc \leq (1-\beta)^{-1} \|\BB_\xi f\|_\spc.
		\end{align*}
		We can now show the perturbation $v$ is $\lin_{\xi}$-bounded:
		\begin{align*}
		\|v f\|_\spc &\leq \nu_0^{-1} \|(\nu + i v \cdot \xi)f\|_\spc \\
		& \leq \frac{1}{\nu_0(1-\beta)} \|\BB_\xi f\|_\spc \\
		& \leq \frac{1}{\nu_0(1-\beta)} \left(\|\lin_{\xi} f\|_\spc + \|\AA_\delta f\|_\spc \right).
		\end{align*}
		We thus have a control $\|v f\|_\spc \leq C\left(\|\lin_{\xi} f\|_\spc + \|f\|_\spc\right)$, where $C=C(\spc, \delta)$. Thanks to this uniform bound, we know (again, by \cite[IV-1.1]{Kato}) that for any~$\xi_0 \in \R^d$ and~$\xi \in \R^d$ satisfying $|\xi-\xi_0| < 1/C$, $\lin_{\xi}$ is closed if $\lin_{\xi_0}$ is. Since $\lin$ is closed, we deduce that $\lin_{\xi}$ is closed for all $\xi \in \R^d$. By the same reasoning and the second line of the previous sequence of estimates, we can show that $\BB_\xi$ is closed for all $\xi \in \R^d$.
		
		Finally, $\dis_\xi+\ah_1$ is m-dissipative because $\dis_\xi$ is boundedly invertible:
		$$\dis_\xi = \left(-1 + \overline{\BB_\delta} \left(\nu + i v \cdot \xi\right)^{-1}\right)\left(\nu + i v \cdot \xi\right),$$
		where $\left\|\overline{\BB_\delta} \left(\nu + i v \cdot \xi\right)^{-1}\|_{\BBB(\spc)} \leq \beta \|\frac{\nu}{\nu + i v \cdot \xi}\right\|_{\BBB(\spc)} = \beta < 1$. Lemma \ref{prop:def_lin_xi} is proved.
	\end{proof}
	
	\begin{notation}
		\label{not:spc_ah}
		In the rest of this paper, we fix some $k > k_*$ and denote the functional spaces
		\begin{equation*}
		\spcp:=L^2\left(\langle v \rangle^k  \right), ~
		\spcg := L^2\left(M^{-1/2}\right), ~ \spc = \spcg \text{ or } \spcp.
		\end{equation*}
		When considering $\lin_\xi$ on one of these spaces, we denote its resolvents
		\begin{gather*}
		\RR(\lambda, \xi) := \left(\lambda - \lin_\xi\right)^{-1}, \\
		\RR(\lambda) := \left(\lambda - \lin\right)^{-1} = \RR(\lambda, 0).
		\end{gather*}
		We also fix some $\ah \in \left(0, \min\{\ah_0, \ah_1\}\right)$, where $\ah_0$ is that of Theorem \ref{thm:hilbert_decomposition} and $\ah_1$ of Lemma \ref{prop:def_lin_xi}.
	\end{notation}
	
	\subsection{Spectral gap properties of \texorpdfstring{$\lin_{\xi}$}{Lξ}}
	\label{scn:spetral_gaps}
	In this section, we show the existence of spectral gaps uniform in $\xi \in \R^d$. More precisely, for small $\xi$, $\Delta_{-\ah} := \{\real z > -\ah\}$ contains a finite amount of eigenvalues converging to zero, and lying in the interior of a fixed closed path $\Gamma$. For large $\xi$, the half plane $\Delta_{-\beh} = \{\real z > -\beh\}$ contains no spectral value, for some constant $\beh > 0$.
	
	After establishing some basic results on the resolvent $\RR(\lambda, \xi)$ (Proposition \ref{prop:analicity_resolvent}), we prove a spectral gap property (Proposition \ref{prop:enlargement}) using the decomposition from Theorem \ref{thm:hilbert_decomposition} for the case $\spc = \spcg$ and an enlargement result from \cite{GMM} to extend it to the case $\spc=\spcp$. The eigenvalues on the right-hand side of this gap are shown to be separated from the rest of the spectrum by a closed path $\Gamma$ (Lemma \ref{prop:isolation_eigenvalues}). We conclude by proving the spectral gap property for large $\xi$.
	
	\begin{prop}
		\label{prop:analicity_resolvent}
		Let $\spc$ be one of the spaces $\spcg$ or $\spcp$.
		For any $\lambda_0 \in \Rho(\lin_{\xi_0})$, we have that~$v \RR(\lambda_0, \xi_0) \in \BBB(\spc)$ and the following expansion around $(\lambda_0, \xi_0)$ holds:
		\begin{gather}
		\label{eqn:second_neumann_exp}
		\RR(\lambda, \xi) = \RR(\lambda_0, \xi_0) \sum_{n=0}^\infty \Big[ (\lambda_0-\lambda)\RR(\lambda_0, \xi_0) + i v \cdot (\xi_0-\xi)\RR(\lambda_0, \xi_0)\Big]^n,
		\end{gather}
		whenever $|\xi-\xi_0| \left\|v\RR(\lambda_0, \xi_0)\right\|_{\BBB(\spc)} + |\lambda-\lambda_0| \left\|\RR(\lambda_0, \xi_0)\right\|_{\BBB(\spc)} < 1$. In particular, the resolvent is continuous on the following set, which is open:$$\left\{(\lambda, \xi) \in \Comp \times \R^d ~| ~\lambda \in \Rho(\lin_\xi)\right\}.$$
		
	\end{prop}
	
	\begin{proof}
		Recall that for some $C > 0$, we have
		$$\|v f\|_\spc \leq C\left(\|\lin_{\xi} f\|_\spc + \|f\|_\spc\right), ~ f \in \dom(\lin_\xi).$$
		We deduce that for any $(\lambda_0, \xi_0) \in \Comp \times \R^d$ such that $\lambda_0 \in \Rho(\lin_{\xi_0})$
		$$\|v \RR(\lambda_0, \xi_0) f\|_\spc \leq C\left(\|\lin_{\xi_0} \RR(\lambda_0, \xi_0)f\|_\spc + \|\RR(\lambda_0, \xi_0)f\|_\spc\right).$$
		Rewriting $\lin_{\xi_0} \RR(\lambda_0, \xi_0) = -1 + \lambda_0  \RR(\lambda_0, \xi_0)$, we have for some constant $C'=C'(\lambda_0, \xi_0)$ that
		\begin{equation}
		\label{eqn:perturbation_resolvent_bound}
		\|v \RR(\lambda_0, \xi_0) f\|_\spc \leq C'\|f\|_\spc.
		\end{equation}
		This means that $v \RR(\lambda_0, \xi_0) \in \BBB(\spc)$. Now, rewrite the resolvent as
		\begin{align*}
		\RR(\lambda, \xi) &= \left(\lambda - \lin_\xi \right)^{-1} \\
		&=\left(\lambda_0- \lin_{\xi_0} - (\lambda_0-\lambda) - i v \cdot(\xi_0 - \xi)\right)^{-1} \\
		&= \RR(\lambda_0, \xi_0) \left(1 - (\lambda_0-\lambda) \RR(\lambda_0, \xi_0) - i (\xi_0-\xi) \cdot v \RR(\lambda_0, \xi_0)\right)^{-1}
		\end{align*}
		whenever $(\lambda, \xi)$ is close enough to $(\lambda_0, \xi_0)$. In such a case, we have the Neumann expansion \eqref{eqn:second_neumann_exp}.
	\end{proof}
	
	\begin{prop}
		\label{prop:enlargement}
		Denote $\ling_\xi = \left(\lin_\xi\right)_{| \spcg}$ and $\linp_\xi = \left(\lin_\xi\right)_{| \spcp}$.
		For any $\xi \in \R^d$, the following set consists of a finite amount of discrete eigenvalues: $$D := \Sigma(\ling_\xi) \cap \overline{\Delta_{-\ah}} = \Sigma(\linp_\xi) \cap \overline{\Delta_{-\ah}},$$ 
		and for any such eigenvalue $\lambda \in D$, we have
		\begin{gather*}
		\left(\Pi_{\linp_\xi, \lambda}\right)_{|\spcg}= \Pi_{\ling_\xi, \lambda},\\
		\notag\ran\left(\Pi_{\linp_\xi, \lambda}\right)=\ran\left(\Pi_{\ling_\xi, \lambda}\right), \\
		\notag\nul(\linp_\xi - \lambda)=\nul(\ling_\xi - \lambda).
		\end{gather*}
		Finally, the following factorization formula holds on $\overline{\Delta_{-\ah}} - D$:
		\begin{equation}
		\label{eqn:fact_interspace}
		\RR_{\linp_\xi} = \RR_{\BB_\xi} + \RR_{\ling_\xi} \AA \RR_{\BB_\xi}.
		\end{equation}
	\end{prop}
	
	\begin{rem}
		The previous proposition means that $\Rho(\lin_\xi) \cap \overline{\Delta_{-\ah}}$, $\Pi_{\lin_\xi, \lambda}$, $\ran\left(\Pi_{\lin_\xi, \lambda}\right)$ and $\nul(\lin_\xi - \lambda)$ can be considered without ambiguity on the space we are working with (the spectral projectors can be restricted to $\spcg$ or extended to $\spcp$ by density).
	\end{rem}
	
	\begin{proof}[Proof of Proposition \ref{prop:enlargement}]
		This is a direct application of \cite[Theorem 2.1]{GMM} whose assumptions are met by Lemma \ref{prop:def_lin_xi}, except for the fact that $\Sigma(\ling_\xi) \cap \Delta_{-\ah}$ is made up of a finite amount of discrete eigenvalues, which is proven below.
		
		For any $\lambda \in \Delta_{-\nu_0}$ such that $\left\|K \RR_{-(\nu + i v \cdot \xi)}(\lambda)\right\|_{\BBB(\spcg)} < 1$, with $K$ from Theorem~\ref{thm:hilbert_decomposition}, the following factorization holds:
		\begin{equation}
		\label{eqn:fact_ling_xi_lambda}
		\RR_{L_\xi}(\lambda) = \RR_{-(\nu + i v \cdot \xi)}(\lambda) \left(1 - K \RR_{-(\nu + i v \cdot \xi)}(\lambda)\right)^{-1}.
		\end{equation}
		The following lemma from S. Ukai and T. Yang \cite[Proposition 2.2.6]{UY} allows to get such estimates for $K \RR_{-(\nu + i v \cdot \xi)}$.
		
		\begin{lem}
			For any $\delta > 0$, we have
			\begin{equation}
			\label{eqn:est_k_nu}
			\sup_{\sigma \geq -\nu_0 + \delta} \left\|K \RR_{-(\nu + i v \cdot \xi)}(\sigma+i \tau)\right\|_{\BBB(\spcg)} \underset{|\xi| + |\tau| \to \infty}{\longrightarrow} 0.
			\end{equation}
		\end{lem}

		Therefore, by estimate \eqref{eqn:est_k_nu}, for any $\xi \in \R^d$, there exists $T_\xi > 0$ such that we have $\overline{\Delta_{-\ah}} \cap \{|\ima z| \geq T_\xi \} \subset \Rho(\ling_\xi)$. Furthermore, as $\ling$ is a non-positive self-adjoint operator according to Theorem~\ref{thm:hilbert_decomposition}, and $iv \cdot \xi$ is skew-symmetric, $\Sigma\left(\ling_\xi\right) \subset (-\infty, 0]$, and thus $$\Sigma(\ling_\xi) \cap \overline{\Delta_{-\ah}} \subset [-\ah, 0] + i [-T_\xi, T_\xi].$$
		However, as $\ling_\xi$ is the sum of a compact operator and the multiplication operator by~$\nu + i v \cdot \xi$, whose range does not meet $\overline{\Delta_{-\ah}}$, \cite[Theorem IV-5.35]{Kato} tells us that~$\Sigma(\ling_\xi) \cap \overline{\Delta_{-\ah}} \subset \spdis(\ling_\xi)$.
		
		In conclusion, $\Sigma(\ling_\xi) \cap \overline{\Delta_{-\ah}}$ is a compact discrete set, thus finite, which yields the conclusion.
	\end{proof}
	
	\begin{lem}
		\label{prop:isolation_eigenvalues}
		There exists $r_0 > 0$ and a closed simple curve $\Gamma$ such that for $|\xi| \leq\,r_0$, the part $\Sigma(\lin_\xi) \cap \Delta_{-\ah}$ is made up of a finite amount of eigenvalues, and these are enclosed by $\Gamma$ which does not meet $\Sigma(\lin_\xi)$.
	\end{lem}
	
	\begin{proof}
		According to Proposition \ref{prop:enlargement}, we can work with $\spc = \spcg$ as $\Sigma(\lin_\xi) \cap \Delta_{-\ah}$ does not depend on the choice of space $\spc$. Recall from Proposition~\ref{prop:analicity_resolvent} that if $\lambda \in \Rho(\lin)$, for any $\xi \in \R^d$ such that $|\xi| \|v \RR(\lambda)\|_{\BBB(\spc)} < 1$, we have $\lambda \in \Rho(\lin_\xi)$.
		
		\step{1}{A control for $v \RR(\lambda)$}
		When $\spc=\spcg$, the resolvent can also be factored on~$\overline{\Delta_{-\ah}} - \{0\}$ as
		\begin{equation}
		\label{eqn:res_gauss_fact}
		\RR = \RR_{-\nu} + \RR_{-\nu} K \RR.
		\end{equation}
		Noting that for any $z \in \overline{\Delta_{-\ah}}$, we have the uniform bound $$\|v \RR_{-\nu}(z)\|_{\BBB(\spcg)} = \sup_{v \in \R^d} \frac{|v|}{|\nu + z|} \leq \frac{1}{\nu_0 - \ah},$$
		and that $\Sigma(\lin) \cap \Delta_{-\ah} = \{0\}$, where $0$ is a semi-simple eigenvalue, and thus a simple pole of the resolvent $\RR$, the following control holds for some $A, B > 0$:
		$$\|v \RR(\lambda)\|_{\BBB(\spcg)} \leq A + \frac{B}{|\lambda|}, ~ \lambda \in \Delta_{-\ah} - \{0\}.$$
		
		\step{2}{Isolation of the eigenvalues}
		By the observation made at the beginning of the proof, if $\lambda \in \Delta_{-\ah} - \{0\}$ and $\xi$ are such that $|\xi|\left(A + B/|\lambda|\right) < 1$, we then have~$\lambda \in \Rho(\lin_\xi)$. In other words, if $\lambda \in \Sigma(\lin_\xi)$, then $|\lambda| \leq B |\xi| (1 - A|\xi|)^{-1}$, therefore, for some $M > 0$ and small enough $r > 0$,
		\begin{equation*}
		\Sigma(\lin_{\xi}) \cap \Delta_{-\ah} \subset \{|\lambda| \leq M |\xi|\}, ~ |\xi| \leq r.
		\end{equation*}
		Choosing some $r_0>0$ small enough, we can consider a closed path $\Gamma$ circling the eigenvalues in $\Sigma(\lin_{\xi}) \cap \Delta_{-\ah}$ for $|\xi| \leq r_0$ while staying in $\Delta_{-\ah}$ (see Figure \ref{fig:isolation_lambda_group}).
	\end{proof}
	
	\begin{figure}
		\includegraphics[width=450pt]{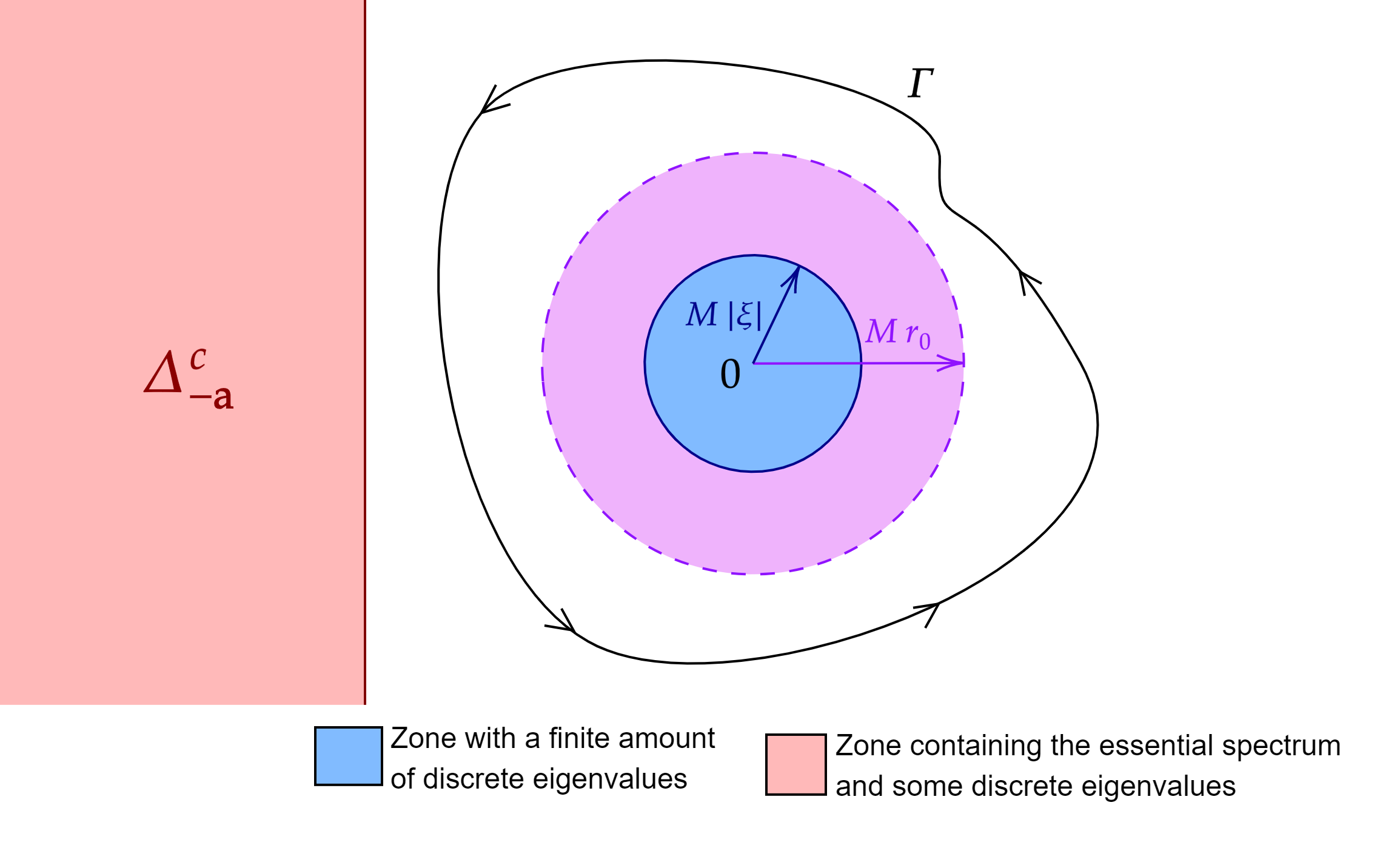}
		\caption{For $|\xi| \leq r_0$, the eigenvalues in $\Delta_{-\ah}$ are confined within a disk of radius $M |\xi|$ and centered at $0$. If $M r_0 < \ah$, we can choose a path enclosing the eigenvalues.}
		\label{fig:isolation_lambda_group}
	\end{figure}
	
	\begin{lem}
		\label{prop:spectral_gap_large_xi}
		For any $r_* > 0$, there exists $\beh_* \in (0, \nu_0)$ such that $$\Sigma(\lin_\xi) \cap \Delta_{-\beh^*} = \emptyset, ~ |\xi| \geq r_*.$$
	\end{lem}
	
	\begin{proof}
		Once again, we can consider the case $\spc = \spcg$. By \eqref{eqn:est_k_nu}, for some large enough~$R_* > 0$, we have that $\Sigma(\lin_\xi) \cap \overline{\Delta_{-\ah}} = \emptyset$ whenever $|\xi| \geq R_*$. Consider now~$r_* > 0$ and the sets
		\begin{gather*}
		X:= \left\{(\lambda, \xi) \in \Comp \times \R^d ~| ~  \lambda \in \Sigma(\lin_\xi), ~ \lambda \in \overline{\Delta_{-\ah}}, ~ r_* \leq |\xi| \leq R_* \right\}, \\
		Y := \{z ~:~\real z \geq 0\} \times \{\xi ~: ~ r_* \leq |\xi| \leq R_*\}.
		\end{gather*}
		If we can show that $X$ is compact and does not meet $Y$, then we shall have the conclusion with $\beh_*:= \textnormal{dist}\left(X, Y\right) > 0$.
		
		\step{1}{Compactness of $X$}
		This set is closed by Proposition \ref{prop:analicity_resolvent}. 
		
		Arguing as in the proof of Lemma \ref{prop:enlargement} (but bounding $\xi$ by $R_*$ instead of fixing it when using \eqref{eqn:est_k_nu}), we have for some $T > 0$ that $\overline{\Delta_{-\ah}} \cap \{|\ima z| \geq T \}$ is included in~$\Rho(\ling_\xi)$ for any $|\xi| \leq R_*$. Thus $X$ is compact because it is closed and contained in~$[-\ah, 0] + i [-T, T] \times \{r_* \leq |\xi| \leq R_*\}$.
		
		\step{2}{$X$ does not meet $Y$}
		We know that $X$ is made up of pairs $(\lambda, \xi)$ such that~$\lambda \in \spdis\left(\ling_\xi\right)$ by Proposition \ref{prop:enlargement}, let us now show that for any $(\lambda, \xi) \in X$, we have $\real \lambda < 0$. As $\lin_\xi$ is dissipative, it is enough to show that it has no eigenvalue in $i \R$ for $\xi \neq 0$. Let us argue by contradiction and consider an eigenvalue $i \tau \in i \R$ and an associated (non-zero) eigenfunction decomposed $f=f_0 + f_1 \in \NN \oplus \NN^\bot$.
		
		\medskip
		\noindent
		Suppose $f_1 \neq 0$. As $\real \langle \lin_\xi f, f \rangle\leq -\ah_0 \|f_1\|^2_\spcg < 0$ by the coercivity of~$\lin$ on~$\NN^\bot$, and~$\real \langle \lin_\xi f, f \rangle= \real \left(i \tau \|f\|_\spcg^2\right) = 0$ because $f$ is an eigenfunction, we get a contradiction, therefore~$f \in \NN$. But this would mean that $\lin_\xi f = -i v \cdot \xi f = i \tau f$, which is impossible as~$f \neq 0$ and $\xi \neq 0$.
	\end{proof}
	
	\section{The eigen problem for small \texorpdfstring{$\xi$}{ξ}}
	\setcounter{equation}{0}
	\setcounter{prop}{0}
	\label{scn:eigen_problem}
	
	We show in this section that for small $\xi$, the eigenvalue $0$ of the unperturbed operator $\lin$ splits into several semi-simple eigenvalues $\lambda_{-1}(\xi), \dots, \lambda_2(\xi)$ of the perturbed operator $\lin_\xi$. We also show that these eigenvalues, corresponding spectral projectors and eigenfunctions have Taylor expansions in $|\xi|$ near $|\xi|=0$.
	
	We rely mostly on perturbation theory and draw inspiration from Kato's \textit{reduction process} \cite[Section II-2.3]{Kato}: considering $\PP(\xi) = \sum \Pi_{\lin_\xi, \lambda}$, where the sum is taken over all eigenvalues $\lambda$ in~$\Sigma(\lin_\xi) \cap \Delta_{-\ah}$, and $\UU(\xi)$ an isomorphism of $\spc$ mapping $\NN$ onto~$\ran(\PP(\xi))$, we have
	\begin{equation*}
	\Sigma(\lin_\xi) \cap \Delta_{-\ah} = \Sigma \left(\left(\lin_\xi\right)_{|\ran\left(\PP(\xi)\right)}\right) = \Sigma\left(\UU(\xi)^{-1} \lin_\xi \,  \UU(\xi)_{| \NN} \right).
	\end{equation*}
	The eigenproblem is thus reduced to the one involving the operator $\UU(\xi)^{-1}\lin_\xi \, \UU(\xi)$ on the finite dimensional space $\NN$.
	
	\medskip
	In the following, we present Taylor expansions of $\PP(\xi)$ (Lemma \ref{prop:proj}) and $\UU(\xi)$ (Lemma~\ref{prop:isomorphisms}). We then define the auxiliary operator $$\widetilde{\lin}(\xi) := \frac{1}{|\xi|} \UU^{-1}(\xi)\lin_\xi \, \UU(\xi)_{| \NN} \in \BBB(\NN),$$ show we can assume $\xi=(r, 0, \dots, 0)$, and give a Taylor approximation (Lemmas \ref{prop:auxiliary_operator} and \ref{prop:shape_auxiliary}) so that we may use Kato's theory to solve our eigenproblem.
	
	\begin{lem}
		\label{prop:proj}
		There exists $r_0 > 0$ such that, for $|\xi| \leq r_0$, the projector
		\begin{equation}
		\label{eqn:spectral_projectors_split}
		\PP(\xi):= \sum_{\lambda \in D} \Pi_{\lin_{\xi}, \lambda},
		\end{equation}
		where $D := \Sigma(\lin_\xi) \cap \Delta_{-\ah}$, expands in $\BBB(\spc)$:
		\begin{gather*}
		\PP(\xi) = \PP(0) + |\xi| \PP^{(1)} \left(\widetilde{\xi}\right)  + \PP^{(2)}(\xi), ~ \widetilde{\xi} := \xi/|\xi|,\\
		\PP^{(1)}\left(\widetilde{\xi}\right) = i \PP(0) v \cdot \widetilde{\xi} \SS + i \SS v\cdot \widetilde{\xi} \PP(0) \in \BBB(\spc),\\
		\PP(0) = \Pi_{\lin, 0}, ~ \left\|\PP^{(2)}(\xi)\right\|_{\BBB(\spc)} = o(|\xi|),
		\end{gather*}
		where $\SS$ is the reduced resolvent of $\lin$ at $\lambda = 0$ (see \eqref{eqn:def_reduced_resolvent} for the definition). Furthermore,~$\PP(\xi)_{| \spcg}$ and $\ran\left(\PP(\xi)\right)$ do not depend on the choice of space $\spc=\spcg$ or~$\spc=\spcp$.
	\end{lem}
	
	\begin{proof}
		Consider $r_0 > 0$ and $\Gamma$ from Lemma~\ref{prop:isolation_eigenvalues}. As $\Gamma$ does not meet $\Sigma(\lin_\xi)$ and encloses the eigenvalues in $\Sigma(\lin_\xi) \cap \Delta_{-\ah}$, the projector~\eqref{eqn:spectral_projectors_split} writes for any $|\xi| \leq r_0$
		\begin{equation*}
		\PP(\xi) = \frac{1}{2 i \pi} \int_\Gamma \RR(\lambda, \xi) d\lambda.
		\end{equation*}
		By the estimate \eqref{eqn:perturbation_resolvent_bound}, $\sup_{\lambda \in \Gamma} \left\|v \RR(\lambda)\right\|_{\BBB(\spc)} < \infty$, thus \eqref{eqn:second_neumann_exp} converges absolutely for~$|\xi| \leq~r_0$ small enough:
		\begin{equation*}
		\RR(\lambda, \xi) = \RR(\lambda) \sum_{n=0}^\infty |\xi|^n \left(-i v \cdot \widetilde{\xi} \RR(\lambda)\right)^n.
		\end{equation*}
		By integrating this series along $\lambda \in \Gamma$, we get the expansion of $\PP(\xi)$. The expression of the coefficients comes from the residue Theorem and the fact that $\lambda=0$ is a semi-simple eigenvalue of $\RR(\lambda)$, combined with the expansion \eqref{eqn:resolvent_expansion_eigenvalue}.
		The last point of the lemma comes from Proposition~\ref{prop:enlargement}.
	\end{proof}
	
	\begin{lem}
		\label{prop:isomorphisms}
		There exists $r_0 > 0$ and a family of invertible maps $\UU(\xi) \in \BBB(\spc)$ for any $|\xi| \leq r_0$ such that $\UU(\xi)$ maps $\NN = \ran\left(\PP(0)\right)$ onto $\ran\left(\PP(\xi)\right)$, and $\UU(\xi)_{| \spcg}$ does not depend on the choice $\spc=\spcg$ or $\spc=\spcp$.
		\medskip
		\noindent
		Furthermore, they follow the expansion
		\begin{equation}
		\label{eqn:isomorphism_expansion}
		\UU(\xi) = \Id + |\xi| \UU^{(1)}\left(\widetilde{\xi}\right) + \UU^{(2)}(\xi),
		\end{equation}
		with $\left\|\UU^{(2)}(\xi)\right\|_{\BBB(\spc)} = o(|\xi|)$ and~$\UU^{(1)}\left(\widetilde{\xi}\right)=i \PP(0) v \cdot \widetilde{\xi} \SS - i \SS v \cdot\widetilde{\xi} \PP(0) \in \BBB(\spc)$, where~$\SS$ is the reduced resolvent of $\lin$ at $\lambda = 0$ (see \eqref{eqn:def_reduced_resolvent} for the definition).
	\end{lem}
	
	\begin{proof}
		Kato's process \cite[Section I-4.6]{Kato} shows that whenever two bounded projectors~$P$ and $Q$ are such that $\|P - Q\|_{\BBB(\spc)} < 1$, we can define an invertible map $U$ satisfying the relation $U P U^{-1} = Q$ by
		\begin{gather*}
		U = U' (1-R)^{-1/2} = (1 - R)^{-1/2} U'
		\end{gather*}
		where we have noted
		\begin{gather*}
		R = (P-Q)^{2}, \\
		(1-R)^{-1/2} = \sum_{n=0}^\infty \binom{-1/2}{n}(-R)^n,\\
		U' = QP + (1-Q)(1-P).
		\end{gather*}
		By assuming $r_0 > 0$ to be small enough so that $\|\PP(\xi) - \PP(0)\|_{\BBB(\spc)} < 1$ whenever~$|\xi| \leq~r_0$, we define this way $\UU(\xi)=U$ with $P=\PP(0)$ and $Q = \PP(\xi)$. The existence of the expansion comes from the expansion of $\PP(\xi)$, and the coefficients can be computed from the latter, using the fact that $\PP(0) \SS = \SS \PP(0) = 0$.
		
		\medskip
		\noindent
		The fact that $\UU(\xi)_{| \spcg}$ does not depend on the choice of $\spc$ comes from the last point of Lemma \ref{prop:proj}.
	\end{proof}
	
	\newcommand\canv{\mathbf{e}}
	
	\begin{lem}
		\label{prop:auxiliary_operator}
		The reduced operator defined by
		\begin{equation*}
		\widetilde{\lin}(\xi):= \frac{1}{|\xi|}\UU(\xi)^{-1} \lin_{\xi} \, \UU(\xi) _{| \NN} \in \BB(\NN)
		\end{equation*}
		does not depend on the initial choice of space~$\spc=\spcg, \spcp$, and has a first order Taylor expansion
		\begin{equation}
		\label{eqn:lin_reduced_expansion}
		\widetilde{\lin}(\xi) = -i \PP(0) \widetilde{\xi} \cdot v + |\xi| \PP(0) \widetilde{\xi} \cdot v \SS \widetilde{\xi} \cdot v + o(|\xi|).
		\end{equation}
		
		\medskip
		\noindent
		Furthermore, for any $|\xi| \leq r_0$, its spectrum is related to the one of $\lin_\xi$ by
		\begin{equation}
		\label{eqn:spectra_lin_lin_reduced}
		\Sigma(\lin_\xi) \cap \Delta_{-\ah} = \Sigma\left(|\xi| \widetilde{\lin}(\xi)\right).
		\end{equation}
	\end{lem}
	
	\begin{proof}
		By Lemmas \ref{prop:proj} and \ref{prop:isomorphisms}, the operator $\widetilde{\lin}(\xi)$ is well defined for $\xi \neq 0$, maps~$\NN$ onto itself, and does not depend on the choice of space $\spc$. As $\UU(\xi)$ has a first order Taylor expansion around $\xi = 0$ in~$|\xi|$, we just need to check that the same is true for $\frac{1}{|\xi|}\lin_\xi \PP(\xi)$.
		
		\medskip
		\noindent
		By estimate \eqref{eqn:perturbation_resolvent_bound}, $\sup_{z \in \Gamma} \|v \RR(z)\|_{\BBB(\spc)} < \infty$, thus the series
		$$\lin_\xi \PP(\xi) = \frac{1}{2 i \pi} \int_\Gamma z \RR(z, \xi) dz = \sum_{n = 0}^\infty \frac{|\xi|^n}{2 i \pi}\int_\Gamma z \RR(z) \left(-i v \cdot \widetilde{\xi} \RR(z)\right)^n dz,$$
		converges absolutely in $\BBB(\spc)$ for $|\xi| \leq r_0$ small enough. Using the residue Theorem, the first terms are
		\begin{itemize}
			\item for $|\xi|^0$ : $0$, because $0$ is a semi-simple eigenvalue of $\lin$, and thus a simple pole of $\RR(z)$,
			\item for $|\xi|^1$ : $- i \PP(0) v \cdot \widetilde{\xi} \PP(0)$,
			\item  for $|\xi|^2$ : $\left(\PP(0) v \cdot \widetilde{\xi}\right)^2 \SS + \SS \left(v \cdot \widetilde{\xi} \PP(0)\right)^2$.
		\end{itemize}
		We get \eqref{eqn:lin_reduced_expansion} by combining this expansion with \eqref{eqn:isomorphism_expansion}.
		
		\medskip
		Finally, as $\Gamma$ circles the eigenvalues in $\Sigma(\lin_\xi) \cap \Delta_{-\ah}$ and $\PP(\xi)$ is the spectral projector associated with them, we have $\Sigma\left(\left(\lin_\xi\right)_{|\ran\left(\PP(\xi)\right)}\right) = \Sigma(\lin_\xi) \cap \Delta_{-\ah}$ according to \cite[Theorem III-6.17]{Kato}, and \eqref{eqn:spectra_lin_lin_reduced} holds as $\UU(\xi)$ is an isomorphism mapping $\NN$ onto~$\ran\left(\PP(\xi)\right)$.
	\end{proof}
	
	Before we prove Theorem \ref{thm:intro_spectral}, we need the following lemma that allows to assume~$\xi$ to be of the form $(r, 0,\dots, 0)$ where $r \in [0, r_0]$, and to deal with the fact that we do not know whether or not the eigenvalues~$\lambda_0$ and $\lambda_2$ of this theorem are distinct.
	
	\begin{lem}
		\label{prop:shape_auxiliary}
		For $0 < |\xi| \leq r_0$ and any $O \in \OO(\R^d)$ such that $O \widetilde{\xi} = (1,0,\dots,0)$,
		\begin{equation}
		\label{eqn:assume_r_0_0}
		O \widetilde{\lin}(\xi) = \widetilde{\lin}(|\xi|,0,\dots, 0) O.
		\end{equation}
		Furthermore, there exist $\widetilde{\lambda}_2(r) \in \Comp$ and a $3 \times 3$ matrix $A(r)$ such that the operator~$\widetilde{\lin}(r):= \widetilde{\lin}(r, 0, \dots, 0)$ writes in the basis $\{\varphi_2,\dots, \varphi_d, \varphi_0, \varphi_1, \varphi_{d+1}\}$
		\begin{equation}
		\label{eqn:lin_reduced_form}
		\widetilde{\lin}(r) = \left(
		\begin{matrix}
		\widetilde{\lambda}_2(r) \Id_{d-1} & \textnormal{O}_{(d-1) \times  3} \\ \textnormal{O}_{3 \times (d-1)} & A(r)
		\end{matrix}
		\right).
		\end{equation}
	\end{lem}
	
	\begin{proof}
		Recall that whenever $O \in \OO\left(\R^d\right)$ is such that~$O \widetilde{\xi} = (1,0,\dots, 0)$, we have the relation~$O \lin_{\xi} = \lin_{(|\xi|, 0, \dots, 0)} O$. As $\RR(z, \xi)$, $\PP(\xi)$, $\UU(\xi)$ and $\widetilde{\lin}(\xi)$ are constructed from~$\lin_\xi$, \eqref{eqn:assume_r_0_0} holds.
		
		\step{1}{Block decomposition}
		Let $j \in \{2, \dots, d\}$ and $k \neq j$. Consider the orthogonal symmetry $O: v_j \leftrightarrow - v_j$. Noting that $O \widetilde{\lin}(r) = \widetilde{\lin}(r) O$ and $O \vph_j = - \vph_j$, we have
		\begin{align*}
		\left\langle \widetilde{\lin}(r) \vph_j, \vph_k \right\rangle_\spc &= \left\langle O \widetilde{\lin}(r) \vph_j, O\vph_k \right\rangle_\spc \\
		& = \left\langle \widetilde{\lin}(r) O \vph_j, O\vph_k \right\rangle_\spc  \\
		& = -\left\langle \widetilde{\lin}(r) \vph_j, \vph_k \right\rangle_\spc.
		\end{align*}
		Therefore, $\left\langle \widetilde{\lin}(r) \vph_j, \vph_k \right\rangle_\spc = 0$, and similarly $\left\langle \widetilde{\lin}(r) \vph_k, \vph_j \right\rangle_\spc = 0$. We conclude that~$\widetilde{\lin}(r)$ has the matrix representation
		\begin{equation*}
		\left(
		\begin{matrix}
		B(r) & \textnormal{O}_{(d-1) \times  3} \\ \textnormal{O}_{3 \times (d-1)} & A(r)
		\end{matrix}
		\right),
		\end{equation*}
		where $B(r)$ is some diagonal $(d-1) \times (d-1)$ matrix.
		
		\step{2}{The diagonal block}
		Consider the orthogonal symmetry $O: v_j \leftrightarrow v_{j+1}$ where~$j \in \{2, \dots, d-1\}$. Noting that $O \widetilde{\lin}(r) = \widetilde{\lin}(r) O$, we have
		\begin{align*}
		\left\langle \widetilde{\lin}(r) \vph_j, \vph_j \right\rangle_\spc &= \left\langle O \widetilde{\lin}(r) \vph_j, O\vph_j \right\rangle_\spc \\
		&= \left\langle \widetilde{\lin}(r) O\vph_j, O\vph_j \right\rangle_\spc \\
		&= \left\langle \widetilde{\lin}(r) \vph_{j+1}, \vph_{j+1} \right\rangle_\spc.
		\end{align*}
		We then conclude to \eqref{eqn:lin_reduced_form} by induction on $j$.
	\end{proof}
	
	\begin{proof}[Proof of Theorem \ref{thm:intro_spectral}]
		The first steps of the proof rely on the following formulas which come from \eqref{eqn:lin_reduced_expansion} and \eqref{eqn:lin_reduced_form}
		\begin{align*}
		\widetilde{\lin}(r) &= \left(
		\begin{matrix}
		\widetilde{\lambda}_2(r) \Id_2 & \textnormal{O}_3 \\ \textnormal{O}_2 & A(r)
		\end{matrix}
		\right) \\
		&= -i \PP(0) v_1 + r \PP(0) v_1 \SS v_1 + o(r).
		\end{align*}
		
		\step{1}{The multiple eigenvalue}
		The operator $\widetilde{\lin}(r)$ has an obvious $(d-1)$-dimensional eigenvalue~$\widetilde{\lambda}_2(r)$. The corresponding eigenvectors are $\vph_2, \dots, \vph_d$, and as they are normalized for the inner product of $\spcg$, we have
		\begin{align*}
		\widetilde{\lambda}_2(r) &= - i \langle v_1 \vph_2, \vph_2 \rangle_\spcg + r \langle \SS v_1 \vph_2, v_1 \vph_2 \rangle_\spcg + o(r) \\
		&= r \langle \SS v_1 \vph_2, v_1 \vph_2 \rangle_\spcg + o(r),
		\end{align*}
		because $v_1 \vph_2^2 M^{-1}(v)$ is odd in $v_1$. The first order derivative $\langle \SS v_1 \vph_2, v_1 \vph_2 \rangle_\spcg$ is negative because $v_1 \varphi_2 \notin \NN$ and $\langle \SS f, f \rangle_\spcg \leq - \ah \|f\|_\spcg$ for any $f \in \NN^\bot$.
		
		\step{2}{The simple eigenvalues}
		We are now going to investigate the eigenvalues of~$\widetilde{\lin}(r)$ on the subspace $X:=\{\vph_0, \vph_1, \vph_{d+1}\}$, that is to say, we are going to study~$A(r)$. We have that
		\begin{gather*}
		A(r) = -i \PP(0) v_1 + r \PP(0) v_1 \SS v_1 + o(r) \text{ on } X.
		\end{gather*}
		The matrix representation of~$A(0)$ is
		\begin{gather*}
		A(0) = -i \left(
		\begin{matrix}
		0 & 1 & 0 \\
		1 & 0 & \sqrt{2/d} \\
		0&\sqrt{2/d}&1
		\end{matrix}
		\right).
		\end{gather*}
		One can show that $A(0)$ is diagonalizable with the following eigenvalues and corresponding eigenvectors
		\begin{align*}
		&i \sqrt{1+2/d}  &&\psi_{-1} = \displaystyle \left(1 + v_1 + \frac{1}{d} \left(|v|^2-d\right) \right) M, \\
		&0 &&\psi_{0} = \displaystyle  \left(1 - \frac{1}{2} \left(|v|^2-d\right) \right)M, \\
		&-i \sqrt{1+2/d}  &&\psi_{1} = \displaystyle  \left(1- v_1 + \frac{1}{d}\left(|v|^2-d\right) \right) M.
		\end{align*}
		By \cite[Theorem II-5.4]{Kato}, $A(r)$ is diagonalizable with three distinct simple eigenvalues~$\widetilde{\lambda}_j(r)= i j \sqrt{1+2/d} + \beta_j r + o(r)$ for~$r$ small enough, where
		\begin{equation*}
		\beta_j = \left\langle \SS v_1 \psi_j, v_1 \psi_j \right\rangle_\spcg  < 0,
		\end{equation*}
		because $v_1 \psi_j \notin \NN$.
		
		Denoting $\lambda_j(|\xi|):= |\xi| \widetilde{\lambda}_j(|\xi|)$, we have \eqref{eqn:eigenvalue_expansion}, and~\eqref{eqn:spectral_gap_small_freq} using the relation \eqref{eqn:spectra_lin_lin_reduced}.
		The spectral gap property \eqref{eqn:spectral_gap_large_freq} is just Lemma~\ref{prop:spectral_gap_large_xi}. Point \textbf{(1)} is proved.
		
		\step{3}{The spectral decomposition}
		We have the decomposition
		\begin{gather*}
		\sum_{j=-1}^2 \widetilde{\PP}_j(r) = \Id_\NN, ~ \widetilde{\PP}_j(r) \widetilde{\PP}_k(r) = \delta_{j, k} \widetilde{\PP}_j(r), \\
		\widetilde{\lin}(r) = \sum_{j=-1}^2 \widetilde{\lambda}_j(r) \widetilde{\PP}_j(r),
		\end{gather*}
		where $\widetilde{\PP}_j(r)$ is the one-dimensional spectral projector of $A(r)$ associated with $\widetilde{\lambda}_j(r)$ and extended by $0$ on $\Span\left(\vph_2,\dots, \vph_d\right)$, and $\widetilde{\PP}_2(r)$ is the projection on $\Span\left(\vph_2, \dots, \vph_d\right)$ parallel to $X$.
		
		\medskip
		By \eqref{eqn:assume_r_0_0}, we go back to the general case of $\xi$ not necessarily of the form $(r, 0, \dots, 0)$, using $O \in \OO(\R^d)$ such that $O \widetilde{\xi} = (1,0,\dots,0)$:
		\begin{align*}
		\lin_\xi \PP(\xi) &= |\xi| \UU(\xi) O \widetilde{\lin}(|\xi|) O^{-1}  \UU(\xi)^{-1} \\
		&=\sum_{j=-1}^2 \lambda_j(|\xi|) \PP_j(\xi),
		\end{align*}
		where we have defined $\PP_j(\xi) := \UU(\xi) O \widetilde{\PP}_j(|\xi|) O^{-1} \UU(\xi)^{-1}$. By Lemma \ref{prop:isomorphisms}, $\UU(\xi)$ has a first order expansion in $\BBB(\spcg)$ and $\UU(\xi)^{-1}$ has one in $\BBB(\spcp)$, therefore this projector has the expansion \eqref{eqn:projector_expansion} in $\BBB(\spcp, \spcg)$, and $\PP_j^{(0)}\left(\widetilde{\xi}\right) = O \widetilde{\PP}_j^{(0)} O^{-1}$. We have thus proved~\eqref{eqn:lin_xi_p_j_xi}-\eqref{eqn:proj_orthogonal}, and \eqref{eqn:sum_proj_0_pi_0} comes from the definition of $\PP(\xi)$ in the case~$|\xi| = 0$.
		
		\step{5}{Range of the projectors for $|\xi| = 0$}
		For $j=0, \pm 1$, $\PP_j^{(0)}\left(\widetilde{\xi}\right)$ is a projection onto the subspace spanned by $e_j^{(0)}\left(\widetilde{\xi}\right)$, where
		
		\begin{gather*}
		e_0^{(0)}\left(\widetilde{\xi}\right) := O^{-1} \psi_0 = \left(1 - \frac{1}{2} \left(|v|^2-d\right) \right)M, \\
		e_{\pm 1}^{(0)}\left(\widetilde{\xi}\right) := O^{-1} \psi_{\pm 1} = \displaystyle  \left(\pm \widetilde{\xi} \cdot v + \frac{1}{d} \left(|v|^2-d\right) \right) M,
		\end{gather*}
		and $\PP_2^{(0)}\left(\widetilde{\xi}\right)$ is a projection on the subspace
		\begin{gather*}
		\Span \left( e_2^{(0)} \left(\widetilde{\xi}\right), \dots, e_d^{(0)} \left(\widetilde{\xi}\right) \right) = \left\{ c \cdot v M ~| ~ c \cdot \widetilde{\xi} = 0 \right\}, \\
		e_j^{(0)} \left(\widetilde{\xi}\right) := O^{-1} \vph_j = C_j\left(\widetilde{\xi}\right) \cdot v M, ~ j=2,\dots,d,
		\end{gather*}
		where $\left(\widetilde{\xi}, C_2\left(\widetilde{\xi}\right), \dots, C_d\left(\widetilde{\xi}\right)\right)$ is an arbitrary orthonormal basis of $\R^d$.
		Point \textbf{(2)} is proved.
		
		\step{6}{Expression of the projectors}
		Consider $\left\{e_j(\xi)\right\}_{j=2}^d$ the family obtained by the Gram-Schmidt orthogonalization of $\left\{ \PP_j(\xi) e_j^{(0)} \left(\widetilde{\xi}\right) \right\}_{j=2}^d$ for the inner product of~$\spcp$, and denote $e_j(\xi) := \PP_j(\xi) e_j^{(0)}\left(\widetilde{\xi}\right)$ for $j=0, \pm 1$. Note that by \eqref{eqn:projector_expansion}, the function $\PP_j(\xi) e_j^{(0)} \left(\widetilde{\xi}\right)$ follows itself an expansion of the form \eqref{eqn:expansion_e_j} with the same~$e_j^{(0)} \left(\widetilde{\xi}\right)$, and since this family is orthogonal for $|\xi|=0$:
		\begin{align*}
		\forall \, 2 \leq j < k \leq d, ~ \left\langle e_j^{(0)} \left(\widetilde{\xi}\right), e_k^{(0)} \left(\widetilde{\xi}\right) \right\rangle_{\spcp} &= \langle O^{-1} \vph_j, O^{-1}\vph_k \rangle_{\spcp}\\ &= \langle \vph_j, \vph_k \rangle_{\spcp}=0,
		\end{align*}
		and the orthogonalization process is smooth, the $e_j(\xi)$ satisfy \eqref{eqn:expansion_e_j}.
		
		Define the functions $f_j(\xi) := \PP_j(\xi)^* e_j(\xi)$ where the adjoint is considered for the inner product of $\spcp$. They have the expansion \eqref{eqn:expansion_f_j} by \eqref{eqn:projector_expansion}. They satisfy the biorthogonality relation \eqref{eqn:biorthogonal} by \eqref{eqn:proj_orthogonal} when $j$ or $k=0, \pm 1$, and by the orthogonalization when $j, k \in \{2, \dots, d\}$.
		
		Point \textbf{(3)} is proved.
	\end{proof}
	
	\begin{rem}
		The coefficients~$C_j\left(\widetilde{\xi}\right)$ in Theorem \ref{thm:intro_spectral} can therefore be assumed to be measurable, but not continuous as it is a non-vanishing tangent vector field on the sphere $\Sp^{d-1}$, by the hairy ball theorem.
	\end{rem}
	
	\section{Exponential decay of the semigroup}
	\setcounter{equation}{0}
	\setcounter{prop}{0}
	\label{scn:splitting_semigroup}
	
	\begin{proof}[Proof of Theorem \ref{thm:intro_splitting_semigroup}]
		The proof will use the following factorization in $\BBB(\spc)$ that comes from the combination of~\eqref{eqn:fact_interspace} and \eqref{eqn:fact_ling_xi_lambda}:
		\begin{equation}
		\label{eqn:fact_general}
		\RR_{\lin_\xi}(z) = \RR_{\BB_\xi}(z) + \RR_{-(\nu + i v\cdot \xi)}(z) \left(1-K \RR_{-(\nu + i v\cdot \xi)(z)}\right)^{-1} \AA \RR_{\BB_\xi}(z)
		\end{equation}
		which holds whenever $\|K \RR_{-(\nu + i v\cdot \xi)}(z)\|_{\BBB(\spcg)} < 1$.
		
		\step{1}{Global estimates}
		Note that as $\lin_\xi - \omega$ is dissipative for $\omega = -\ah + \|\AA\|_{\BBB(\spc)}$ according to Lemma \ref{prop:def_lin_xi} and the fact that $i v \cdot \xi$ is skew-symmetric, we have
		\begin{equation*}
		\forall \xi \in \R^d, ~ \|S_{\lin_\xi}(t)\|_{\BBB(\spc)} \leq e^{\omega t}.
		\end{equation*}
		Furthermore, \eqref{eqn:est_k_nu} means that for some $T > 0$, $\|K \RR_{-(\nu + i v\cdot \xi)}(z)\|_{\BBB(\spcg)} \leq 1/2$ holds if~$|\ima z | \geq T$ and $\real z \geq -\ah$. The factorization \eqref{eqn:fact_general} combined with the dissipativity of~$\dis_\xi+\ah_1$ from Lemma \ref{prop:def_lin_xi} yields the following bound for $z \in \Delta_{-\ah} \cap \left\{|\ima z| \geq T\right\}$ and $\xi \in \R^d$:
		\begin{align*}
		\left\|\RR_{\lin_\xi}(z) \right\|_{\BBB(\spc)} &\leq 
		\left\|\RR_{\BB_\xi}(z)\right\|_{\BBB(\spc)}  + 2\left\|\RR_{-(\nu + i v\cdot \xi)}(z)\right\|_{\BBB(\spc)} \left\|\AA \RR_{\BB_\xi}(z)\right\|_{\BBB(\spc, \spcg)} \\
		& \leq (\ah_1-\ah)^{-1} + \frac{2 \|\AA\|_{\BBB(\spc, \spcg)}}{(\nu_0 - \ah)(\ah_1 - \ah)} \leq M,
		\end{align*}
		for some $M > 0$. The dissipativity of $\lin_\xi - \omega$ tells us that, taking $T$ large enough, we may assume 
		\begin{equation}
		\label{eqn:glob_est}
		\forall \xi \in \R^d, z \in U, ~ \left\|\RR_{\lin_\xi}(z) \right\|_{\BBB(\spc)} \leq M,
		\end{equation}
		where $U := \Delta_{-\ah} \cap \{|z + \ah| \geq T\}$.
		
		\step{2}{Small $\xi$}
		Recall from Lemma \ref{prop:isolation_eigenvalues} and Theorem \ref{thm:intro_spectral} that $r_0$ was chosen small enough so that for some $\delta > 0$,
		\begin{equation*}
		\Sigma(\lin_\xi) \cap \Delta_{-\ah + 2\delta} = \Sigma(\lin_\xi) \cap \Delta_{-\ah} = \{\lambda_{-1}(|\xi|), \dots, \lambda_2(|\xi|)\}
		\end{equation*}
		whenever $|\xi| \leq r_0$. 	In particular, $-\ah+\delta + i\R \subset \Rho\left(\lin_\xi\right)$ for any $|\xi| \leq r_0$, and by the continuity of~$\RR_{\lin_\xi}(z)$ in $(z, \xi)$ combined with \eqref{eqn:glob_est}, we have for some $K^{(-)}_0 > 0$
		\begin{equation*}
		\label{eqn:small_est_line}
		\forall |\xi| \leq r_0, ~ \sup_{-\ah + \delta + i \R} \|\RR_{\lin_\xi}\|_{\BBB(\spc)} \leq K^{(-)}_0.
		\end{equation*}
		Denote the following invariant subspaces and restriction by
		$$\NN(\xi) := \ran\left(\PP(\xi)\right), ~\NN(\xi)^\bot := \ran\left(1-\PP(\xi)\right),~ \lin^\bot_\xi := \left(\lin_\xi\right)_{| \NN(\xi)^\bot}.$$
		By \cite[Theorem III-6.17]{Kato}, $\Sigma(\lin_\xi^\bot) \cap\Delta_{-\ah} = \emptyset$, and the semigroup and the resolvent associated with $\lin_\xi$ split along the direct sum $\spc = \NN(\xi) \oplus \NN(\xi)^\bot$ as
		\begin{align*}
		S_{\lin_\xi}(t) f = \sum_{j=-1}^2 e^{\lambda_j(|\xi|) t} \PP_j(\xi) f + S_{\lin_\xi^\bot}(t)f^\bot,\\
		\label{eqn:split_res}
		\RR_{\lin_\xi}(z) = \sum_{j=-1}^2 \frac{\PP_j(\xi)}{\lambda_j(|\xi|) - z} f + \RR_{\lin_\xi^\bot}(z) f^\bot,
		\end{align*}
		where $f^\bot=f-\PP(\xi)f$. Using the fact that $\RR_{\lin^\bot_\xi}$ is holomorphic on $\Delta_{-\ah}$ and the maximum modulus principle, we deduce from the relation between $\RR_{\lin_\xi}$ and $\RR_{\lin_\xi^\bot}$, and the previous estimates, the bound
		\begin{equation*}
		\forall z \in \Delta_{-\ah + \delta}, ~\left\|\RR_{\lin^\bot_\xi}(z)\right\|_{\BBB\left(\NN(\xi)^\bot\right)} \leq \max\left\{M, K^{(-)}_0 + 4/\delta\right\} =: K^{(-)},
		\end{equation*}
		which is uniform in $|\xi| \leq r_0$.
		
		\smallskip
		We have shown that for any fixed $|\xi| \leq r_0$, the operator $\lin_\xi^\bot$ satisfies the assumptions of Theorem \ref{thm:gpg} with $X=\NN(\xi)^\bot, \alpha = \omega, C_\alpha = 1, \beta = -\ah + \delta, K_\beta = K^{(-)}$. We thus have the bound 
		\begin{equation*}
		\forall |\xi| \leq r_0, ~ \left\|S_{\lin_\xi^\bot}(t)\right\|_{\BBB(\NN(\xi)^\bot)} \leq C^{(-)} e^{(-\ah + \delta) t},
		\end{equation*}
		for some $C^{(-)} > 0$. For $|\xi| \leq r_0$, we define $\VV(t, \xi)$ to be $S_{\lin_\xi^\bot}(t)$ extended by 0 on~$\NN(\xi)$ (note that it does not change its growth estimate).
		
		\step{3}{Large $\xi$}
		By Lemma \ref{prop:spectral_gap_large_xi}, for some $\beh \in (0, \ah)$, we have
		\begin{equation*}
		\forall |\xi| \geq r_0, ~ \Sigma(\lin_\xi) \cap \overline{\Delta_{-\beh}} = \emptyset.
		\end{equation*}
		By \eqref{eqn:est_k_nu}, we may assume that for some large enough $R > r_0$,
		\begin{equation*}
		\forall |\xi| \geq R, ~ \sup_{\Delta_{-\beh}} \left\|\RR_{\lin_\xi}\right\|_{\BBB(\spc)} \leq M
		\end{equation*}
		also holds. Again, the continuity of $\RR_{\lin_\xi}$ implies the existence of a bound $K_0^{(+)} > 0$ on $-\beh + i \R$ uniform in $|\xi| \geq r_0$, and by a similar argument as in Step 2, we prove
		\begin{equation*}
		\forall |\xi| \geq r_0, \left\|\RR_{\lin_\xi}(z) \right\|_{\BBB(\spc)} \leq K^{(+)}
		\end{equation*}
		for some $K^{(+)} > 0$. We invoke once again Theorem \ref{thm:gpg} with $X=\spc, \alpha = \omega, C_\alpha = 1, \beta = -\beh, K_\beta = K^{(+)}$ to obtain
		\begin{equation*}
		\forall |\xi| \geq r_0, ~ \left\|S_{\lin_\xi}(t)\right\|_{\BBB(\spc)} \leq C^{(+)} e^{-\beh t}
		\end{equation*}
		for some $C^{(+)} > 0$. For $|\xi| \geq r_0$, we define $\VV(t, \xi)$ to be $S_{\lin_\xi}(t)$. We finally get the conclusion with $\gamma=\min\left\{-\ah+\delta, \beh\right\}$ and~$C=\max\{C^{(-)}, C^{(+)}\}$.
	\end{proof}

	\appendix
	\section{}
	\setcounter{equation}{0}
	\setcounter{prop}{0}
	
	\subsection{Spectral theory}
	\label{scn:spectral_theory}
	Consider a Banach space $X$ and $\Lambda \in \CCC(X)$. If $\lambda \in \spdis(\Lambda)$,  then $\lambda$ is a finite order pole of the resolvent, which can be expanded as
	\begin{equation}
	\label{eqn:resolvent_expansion_eigenvalue}
	\RR_\Lambda(\lambda + h) = \sum_{k=1}^{m} \frac{D^k}{h^{k+1}} + \frac{1}{h} \Pi_{\Lambda, \lambda} - \sum_{n=0}^\infty h^n S^{n+1}.
	\end{equation}
	The operator $D \in \BBB(X)$ is called the \textit{eigennilpotent} and satisfies
	\begin{align*}
	D^m & = 0, \\
	D \Pi_{\Lambda, \lambda} &= \Pi_{\Lambda, \lambda} D = D, \\
	\Lambda \Pi_{\Lambda, \lambda} &= \Pi_{\Lambda, \lambda} \Lambda = \lambda \Pi_{\Lambda, \lambda} + D.
	\end{align*}
	The operator $S \in \BBB(X)$ is called the \textit{reduced resolvent} and satisfies
	\begin{align}
	\label{eqn:def_reduced_resolvent}
	S f &= \left\{
	\begin{matrix}
	0, & f \in \ran(\Pi_{\Lambda, \lambda}), \\
	-\left(\lambda-\Lambda\right)^{-1}f, & f \in \ran(1 - \Pi_{\Lambda, \lambda}),
	\end{matrix}
	\right.\\
	\notag
	S \Lambda &\subset \Lambda S = 1 - \Pi_{\Lambda, \lambda}.
	\end{align}
	The eigenvalue $\lambda$ is said to be \textit{semi-simple} when both eigenspaces are equal, or equivalently when the eigennilpotent is zero (which is the same as saying the eigenvalue is a pole of order~1).
	
	\medskip
	When two closed simple paths $\Gamma_1$ and $\Gamma_2$ with values in the resolvent set of $\Lambda$ are such that $\Gamma_1$ lies in the exterior of $\Gamma_2$, we have
	$$\int_{\Gamma_1} \RR_\Lambda(z) dz \int_{\Gamma_2} \RR_\Lambda(z) dz = 0.$$
	
	For a detailed presentation of these results, see \cite[Section III-6.5]{Kato}.
	
	\subsection{Semigroup theory}
	
	The famous Hille-Yosida Theorem (\textbf{(1) $\Leftrightarrow$ (2)} below, see for example \cite[Chapter 1, Theorem 3.1]{Pazy}) and Lummer-Phillips Theorem (\textbf{(1) $\Leftrightarrow$ (3)} below, \cite[Chapter 1, Theorem 4.3]{Pazy}) give necessary and sufficient conditions for a closed and densely defined operator to be a $\CC^0$-semigroup generator.
	
	\begin{theo}[Hille-Yosida-Lummer-Phillips]
		\label{thm:hylp}
		Let $\Lambda$ be a closed and densely defined operator on a Banach space $X$, the following conditions are equivalent for any~$C > 0$ and $\omega \in \R$:
		\begin{enumerate}[label=\textnormal{\textbf{(\arabic*)}}]
			\item $\Lambda$ generates a $\CC^0$-semigroup satisfying $\|S_\Lambda(t)\|_{\BBB(X)} \leq C e^{\omega t}$,
			\item $\Sigma(\Lambda) \cap \Delta_\omega = \emptyset$ and $\displaystyle \|\RR_\Lambda(z)\|_{\BBB(X)} \leq \frac{C}{|\real z - \omega|}$ for $z \in \Delta_\omega$,
			\item $C \|\left(\Lambda - z\right) f\|_X \geq (z - \omega) \|f\|_X$ for $f \in \dom(\Lambda)$, $z > \omega$, and $\Rho(\Lambda) \cap \Delta_\omega \neq \emptyset$.
		\end{enumerate}
	\end{theo}
	
	Note that when $X$ is a Hilbert space, an m-dissipative operator, that is to say an operator~$\Lambda$ such that
	$$\real \langle \Lambda f, f \rangle_X \leq 0, ~ \Rho(\Lambda) \cap \Delta_0 \neq \emptyset,$$
	satisfies the equivalent conditions of Theorem \ref{thm:hylp} with $C = 1$ and $\omega = 0$.
	
	Furthermore, still in a Hilbert setting, the growth estimate is directly linked to the size of the half-plane on which the resolvent is bounded: we give here a version of \cite[V-Theorem 1.11]{EN} in which we specify the dependency of the constant in the growth estimate.
	\begin{theo}[Gearhart-Prüss-Greine]
		\label{thm:gpg}
		Consider $\Lambda$ a $\CC^0$-semigroup generator on a Hilbert space $X$, satisfying $\|S_\Lambda(t)\|_{\BBB(X)} \leq C_\alpha e^{\alpha t}$, and whose resolvent is defined and uniformly bounded on $\Delta_\beta$ by $K_\beta$. The semigroup satisfies $\|S_\Lambda(t)\|_{\BBB(X)} \leq C_\beta e^{\beta t}$ for some constructive constant $C_\beta > 0$ depending on $K_\beta, C_\alpha, \alpha$ and $\beta$.
	\end{theo}
	
	\bigskip
	\bibliographystyle{abbrv}
	\bibliography{bibliography}

\begin{thebibliography}{10}

\bibitem{BM}
C.~{Baranger} and C.~{Mouhot}.
\newblock {Explicit spectral gap estimates for the linearized Boltzmann and
  Landau operators with hard potentials.}
\newblock {\em {Rev. Mat. Iberoam.}}, 21:819--841, 2005.

\bibitem{BGL}
C.~Bardos, F.~Golse, and D.~Levermore.
\newblock Fluid dynamic limits of kinetic equations. {I}. {F}ormal derivations.
\newblock {\em Journal of Statistical Physics}, 63:323–344, 1991.

\bibitem{BU}
C.~Bardos and S.~Ukai.
\newblock The classical incompressible {N}avier-{S}tokes limit of the
  {B}oltzmann equation.
\newblock {\em Mathematical Models and Methods in Applied Sciences}, pages
  235--257, 1991.

\bibitem{CIP}
C.~Cercignani, R.~Illner, and M.~Pulvirenti.
\newblock {\em The Mathematical Theory of Dilute Gases}.
\newblock Springer, 1994.

\bibitem{EP}
R.~Ellis and M.~Pinsky.
\newblock The first and second fluid approximations of the linearized boltzmann
  equation.
\newblock {\em Journal de Math{\'e}matiques pures et appliqu{\'e}es}, pages
  125--156, 1975.

\bibitem{EN}
K.-J. Engel and R.~Nagel.
\newblock {\em One-parameter semigroups for linear evolution equations}.
\newblock Graduate texts in mathematics. Springer, 2000.

\bibitem{IGT}
I.~Gallagher and I.~Tristani.
\newblock {On the convergence of smooth solutions from Boltzmann to
  Navier-Stokes}.
\newblock {\em {Annales Henri Lebesgue}}, page 561–614, 2019.

\bibitem{G}
F.~Golse.
\newblock {\em Handbook of differential equations}.
\newblock 2005.

\bibitem{Grad}
H.~{Grad}.
\newblock {Asymptotic Theory of the Boltzmann Equation 2}.
\newblock 1:147--181, 1963.

\bibitem{GMM}
M.~P. {Gualdani}, S.~{Mischler}, and C.~{Mouhot}.
\newblock {Factorization for non-symmetric operators and exponential
  H-theorem}.
\newblock {\em M{\'e}moires de la SMF}, 153, 2017.

\bibitem{Hilbert}
D.~Hilbert.
\newblock Begründung der kinetischen gastheorie.
\newblock {\em Mathematische Annalen}, 72:562–577, 1912.

\bibitem{Kato}
T.~Kato.
\newblock {\em Perturbation theory for linear operators}.
\newblock Springer, 1966.

\bibitem{Pazy}
A.~Pazy.
\newblock {\em Semigroups of Linear Operators and Applications to Partial
  Differential Equations}.
\newblock 1974.

\bibitem{UY}
S.~Ukai and T.~Yang.
\newblock Mathematical theory of {B}oltzmann equation. lecture notes
  {S}eries-no. 8, {H}ong {K}ong: Liu {B}ie {J}u {C}enter for {M}athematical
  {S}ciences, {C}ity {U}niversity of {H}ong {K}ong.

\bibitem{YY}
T.~Yang and H.~Yu.
\newblock {S}pectrum analysis of some kinetic equations.
\newblock {\em Archive for Rational Mechanics and Analysis}, 222:731–768,
  2016.

\end{thebibliography}
	
\end{document}